\documentclass[reqno]{amsart}

\usepackage[danish,english]{babel}
\usepackage[latin1]{inputenc}
\usepackage[T1]{fontenc}
\usepackage{amsmath,amsthm,amssymb,amsfonts,mathrsfs,latexsym}
\usepackage{graphicx}
\usepackage{verbatim}
\usepackage[all]{xy}
\usepackage[numbers]{natbib}
\usepackage{hyperref}
\usepackage[shortlabels]{enumitem}
\setenumerate[0]{{\upshape(1)}}

\numberwithin{equation}{section}

\newtheorem{thm}{Theorem}[section]
\newtheorem{lem}[thm]{Lemma}
\newtheorem{prop}[thm]{Proposition}
\newtheorem{cor}[thm]{Corollary}
\newtheorem*{thm*}{Theorem}
\newtheorem{prob}[thm]{Problem}

\theoremstyle{definition}
\newtheorem{defi}[thm]{Definition}

\newtheorem{rem}[thm]{Remark}

\newcommand{\mc}[1]{\mathcal{#1}}

\newcommand{\mb}[1]{\mathbb{#1}}

\newcommand{\mr}[1]{\mathrm{#1}}
\newcommand{\us}[1]{\upshape{#1}}

\newcommand{\C}{\mathbb{C}}
\newcommand{\N}{\mathbb{N}}

\newcommand{\R}{\mathbb{R}}
\newcommand{\Z}{\mathbb{Z}}

\newcommand{\la}{\langle}
\newcommand{\ra}{\rangle}
\renewcommand{\epsilon}{\varepsilon}
\renewcommand{\phi}{\varphi}
\renewcommand{\bar}{\overline}
\renewcommand{\tilde}{\widetilde}

\newcommand{\Cs}{\ensuremath{C^*}}
\newcommand{\F}{\mb F}
\newcommand{\id}{\mr{id}}
\newcommand{\tensor}{\otimes}
\newcommand{\ev}{\mr{ev}}
\newcommand{\Tr}{\mr{Tr}}
\newcommand{\HS}{{B_2}}
\newcommand{\cb}{\mr{cb}}
\newcommand{\conv}{\mr{conv}}
\newcommand{\sym}{\mr{sym}}

\makeatletter
\def\blfootnote{\xdef\@thefnmark{}\@footnotetext}
\makeatother

\begin{document}
\selectlanguage{english}

\begin{abstract}
In order to investigate the relationship between weak amenability and the Haagerup property for groups, we introduce the weak Haagerup property, and we prove that having this approximation property is equivalent to the existence of a semigroup of Herz-Schur multipliers generated by a proper function (see Theorem~\ref{lem:wh}). It is then shown that a (not necessarily proper) generator of a semigroup of Herz-Schur multipliers splits into a positive definite kernel and a conditionally negative definite kernel. We also show that the generator has a particularly pleasant form if and only if the group is amenable.

In the second half of the paper we study semigroups of radial Herz-Schur multipliers on free groups. We prove that a generator of such a semigroup is linearly bounded by the word length function (see Theorem~\ref{thm:B}).
\end{abstract}


\title{Semigroups of Herz-Schur multipliers}
\author{S\o ren Knudby}
\thanks{Supported by ERC Advanced Grant No. OAFPG 247321 and the Danish National Research Foundation through the
Centre for Symmetry and Deformation (DNRF92).\\
E-mail: knudby@math.ku.dk}
\address{Department of Mathematical Sciences, University of Copenhagen,
\newline Universitetsparken 5, DK-2100 Copenhagen \O, Denmark}
\email{knudby@math.ku.dk}

\date{\today}
\maketitle



\section{Introduction}
The notion of amenability for groups was introduced by von Neumann \cite{amenable} and has played an important role in the field of operator algebras for many years. It is well-known that amenability of a group is reflected by approximation properties of the \Cs-algebra and von Neumann algebra associated with the group. More precisely, a discrete group is amenable if and only if its (reduced or universal) group \Cs-algebra is nuclear if and only if its group von Neumann algebra is semidiscrete.

Amenability may be seen as a rather strong condition to impose on a group, and several weakened forms have appeared, two of which are \emph{weak amenability} and the \emph{Haagerup property}. Recall that a discrete group $G$ is amenable if and only if there is a net $(\phi_i)_{i\in I}$ of finitely supported, positive definite functions on $G$ such that $\phi_i\to 1$ pointwise. When the discrete group is countable, which will always be our assumption in this paper, we can of course assume that the net is actually a sequence. We have included a few well-known alternative characterizations of amenability in Theorem~\ref{thm:amenable}.

A countable, discrete group $G$ is called \emph{weakly amenable} if there exist $C > 0$ and a net $(\phi_i)_{i\in I}$ of finitely supported Herz-Schur multipliers on $G$ converging pointwise to 1 and $\|\phi_i\|_{B_2} \leq C$ for all $i\in I$ where $\|\cdot\|_{B_2}$ denotes the Herz-Schur norm. The infimum of all $C$ for which such a net exists, is called the \emph{Cowling-Haagerup constant} of $G$, usually denoted $\Lambda_{\mr{cb}}(G)$.

The countable, discrete group $G$ has the \emph{Haagerup property} if there is a net $(\phi_i)_{i\in I}$ of positive definite functions on $G$ converging pointwise to 1 such that each $\phi_i$ vanishes at infinity. An equivalent condition is the existence of a conditionally negative definite function $\psi:G\to\R$ such that $\psi$ is proper, i.e. $\{ g\in G \mid |\psi(g)| < n\}$ is finite for each $n\in\N$ (see for instance \cite[Theorem 2.1.1]{MR1852148}). It follows from Schoenberg's Theorem that given such a $\psi$, the family $(e^{-t\psi})_{t>0}$ witnesses the Haagerup property.

For a general treatment of weak amenability and the Haagerup property, including examples of groups with and without these properties, we refer the reader to \cite{BO}.

Since positive definite functions are also Herz-Schur multipliers with norm 1, it is clear that amenability is stronger than both weak amenability with (Cowling-Haagerup) constant 1 and the Haagerup property. A natural question to ask is how weak amenability and the Haagerup property are related. For a long time the known examples of weakly amenable groups with constant 1 also had the Haagerup property and vice versa. Also, the groups that were known to not be weakly amenable also failed the Haagerup property. So it seemed natural to ask if the Haagerup property is equivalent to weak amenability with constant 1. This turned out to be false, and the first counterexample was the wreath product $\Z/2 \wr \F_2$. This group is defined as the semidirect product $(\bigoplus_{\F_2} \Z/2) \rtimes \F_2$, where the action $\F_2 \curvearrowright \bigoplus_{\F_2} \Z/2$ is the shift. In \cite{MR2393636} it is shown that the group $\Z/2 \wr \F_2$ has the Haagerup property, and in \cite[Corollary 2.12]{MR2680430} it was proved that $\Z/2 \wr \F_2$ is not weakly amenable with constant 1. In fact, the group is not even weakly amenable \cite[Corollary 4]{MR2914879}.

It is still an open question if groups that are weakly amenable with constant 1 have the Haagerup property. It may be formulated as follows. Given a net $(\phi_i)_{i\in I}$ of finitely supported functions on $G$ such that $\|\phi_i\|_{B_2} \leq 1$ and $\phi_i\to 1$ pointwise, does there exist a proper, conditionally negative definite function on $G$? We do not answer this question here, but we consider the following related problem. If we replace the condition that each $\phi_i$ is finitely supported with the condition that $\phi_i$ vanishes at infinity, what can then be said? We make the following definition.

\begin{defi}
A discrete group $G$ has the \emph{weak Haagerup property} if there exist $C > 0$ and a net $(\phi_i)_{i\in I}$ of Herz-Schur multipliers on $G$ converging pointwise to $1$ such that each $\phi_i$ vanishes at infinity and satisfies $\|\phi_i\|_{B_2} \leq C$. If we may take $C = 1$, then $G$ has the weak Haagerup property \emph{with constant 1}.
\end{defi}

A priori the weak Haagerup property is even less tangible than weak amenability, but the point is that with the weak Haagerup property with constant 1, we can assume that the net in question is a semigroup of the form $(e^{-t\phi})_{t>0}$, as the following holds.

\begin{thm}\label{lem:wh}
For a countable, discrete group $G$, the following are equivalent.
\begin{enumerate}
	\item  There is a sequence $(\phi_n)$ of functions vanishing at infinity such that $\phi_n\to 1$ pointwise and $\|\phi_n\|_{B_2} \leq 1$ for all $n$.
	\item There is $\phi:G \to \R$ such that $\phi$ is proper and $|| e^{-t\phi}||_{B_2} \leq 1$ for every $t>0$.
\end{enumerate}
\end{thm}

The proof of the above is reminiscent of the proof concerning the equivalent formulations of the Haagerup property (see Theorem 2.1.1 in the book \cite{MR1852148}). We provide a proof in Section~\ref{sec:prelim} (see Theorem~\ref{lem:WH}).

Clearly, weak amenability with constant 1 implies the weak Haagerup property with constant 1, and the converse is false by the example $\Z/2 \wr \F_2$ from before. It is also obvious that the Haagerup property implies the weak Haagerup property with constant 1. It is not clear, however, if they are in fact equivalent.

In the light of the previous theorem we consider the following problem.

\begin{prob}\label{prob:2}
Let $G$ be a countable, discrete group, and let $\phi:G\to\R$ be a symmetric function satisfying $\| e^{-t\phi} \|_{B_2} \leq 1$ for all $t> 0$. Does there exist a conditionally negative definite function $\psi$ on $G$ such that $\phi \leq \psi$?
\end{prob}

Note that $\phi$ is proper if and only if each $e^{-t\phi}$ vanishes at infinity. A positive solution to the problem would prove that the Haagerup property is equivalent to the weak Haagerup property with constant 1. So in particular, a solution to Problem~\ref{prob:2} would prove that weak amenability with constant 1 implies the Haagerup property.

We will prove the following theorem.

\begin{thm}\label{thm:A}
Let $G$ be a countable, discrete group with a symmetric function $\phi:G\to\R$. Then $\|e^{-t\phi}\|_{B_2} \leq 1$ for every $t>0$ if and only if $\phi$ splits as
$$
\phi(y^{-1}x) = \psi(x,y) + \theta(x,y) + \theta(e,e) \qquad (x,y\in G),
$$
where
\begin{itemize}
\renewcommand{\labelitemi}{$\cdot$}
	\item $\psi$ is a conditionally negative definite kernel on $G$ vanishing on the diagonal,
	\item and $\theta$ is a bounded, positive definite kernel on $G$.
\end{itemize}
\end{thm}

\noindent
The downside of the above theorem is that the functions $\psi$ and $\theta$ are defined on $G\times G$ instead of simply $G$. A natural question to ask is in which situations we may strengthen Theorem~\ref{thm:A} to produce functions $\psi$ and $\theta$ defined on the group $G$ itself. It is not so hard to prove that this happens if $G$ is amenable. Moreover, this actually characterizes amenability. We thus have following theorem.

\begin{thm}\label{thm:amenable-case}
Let $G$ be a countable, discrete group. Then $G$ is amenable if and only if the following condition holds. Whenever $\phi:G\to\R$ is a symmetric function such that $\|e^{-t\phi}\|_\HS \leq 1$ for every $t>0$, then $\phi$ splits as
$$
\phi(x) = \psi(x) + ||\xi||^2 + \la\pi(x)\xi,\xi\ra  \qquad (x\in G)
$$
where
\begin{itemize}
\renewcommand{\labelitemi}{$\cdot$}
	\item $\psi$ is a conditionally negative definite function on $G$ with $\psi(e) = 0$,
	\item $\pi$ is an orthogonal representation of $G$ on some real Hilbert space $H$,
	\item and $\xi$ is a vector in $H$.
\end{itemize}

\end{thm}

\noindent
Note that the function $x\mapsto \la \pi(x)\xi,\xi\ra$ is positive definite, and every positive definite function has this form.

We solve Problem~\ref{prob:2} in the special case where $G$ is a free group and the function $\phi$ is radial. The result is the following theorem, which generalizes Corollary 5.5 from \cite{HK1}.

\begin{thm}\label{thm:B}
Let $\F_n$ be the free group on $n$ generators ($2\leq n \leq \infty$), and let $\phi:\F_n\to\R$ be a radial function, i.e., $\phi(x)$ depends only on the word length $|x|$. If $\|e^{-t\phi}\|_{B_2} \leq 1$ for every $t>0$, then there are constants $a,b \geq 0$ such that
$$
\phi(x) \leq b + a|x| \qquad\text{for all } x\in\F_n.
$$
\end{thm}

The paper is organized as follows. In Section~\ref{sec:prelim} we introduce many of the relevant notions needed in the rest of the paper. Section~\ref{sec:wh-char} contains the proof of Theorem~\ref{lem:wh}, and Section~\ref{sec:A} contains the proof of Theorem~\ref{thm:A}. Section~\ref{sec:amenable} considers the case of amenable groups. Here we prove Theorem~\ref{thm:amenable-case}.

The proof of Theorem~\ref{thm:B} concerning $\F_n$ takes up the second half of the paper. The proof is divided into two cases depending on whether $n$ is finite or infinite. In Section~\ref{sec:infinite} we deal with the infinite case, and Section~\ref{sec:finite} contains the finite case. The strategy of the proof is to compare the Herz-Schur norm of $e^{-t\phi}$ with the norm of certain functionals on the Toeplitz algebra. This is accomplished in Propositions~\ref{prop:HSS28} and~\ref{prop:norms}. It turns out that a certain norm bound on the functionals produces a splitting of our function $\phi$ into a positive definite and a conditionally negative definite part (Theorem~\ref{thm:split}). Characterizing the positive and conditionally negative parts (Corollary~\ref{cor:4} and Proposition~\ref{prop:thm2}) then leads to Theorem~\ref{thm:B} in the case of $\F_\infty$.

When $n < \infty$, Theorem~\ref{thm:B} is deduced in basically the same way as the case $n = \infty$, but the details are more complicated. The transformations introduced in Section~\ref{sec:FG} allow us to reduce many of the arguments for $\F_n$ with $n$ finite to the case of $\F_\infty$.

\section{Preliminaries}\label{sec:prelim}

Let $X$ be a non-empty set. A \emph{kernel} on $X$ is a function $k:X\times X\to\C$. The kernel is called \emph{symmetric} if $k(x,y) = k(y,x)$ for all $x,y\in X$, and \emph{hermitian} if $k(y,x) = \bar{k(x,y)}$.

The kernel $k$ is said to be \emph{positive definite}, if
$$
\sum_{i,j=0}^n c_i \bar{c_j}k(x_i,x_j) \geq 0
$$
for all $n\in\N$, $x_1,\ldots,x_n\in X$ and $c_1,\ldots,c_n\in\C$. It is called \emph{conditionally negative definite} if it is hermitian and
$$
\sum_{i,j=0}^n c_i \bar{c_j}k(x_i,x_j) \leq 0
$$
for all $n\in\N$, $x_1,\ldots,x_n\in X$ and $c_1,\ldots,c_n\in\C$ such that $\sum_{i=0}^n c_i = 0$.

Recall Schoenberg's Theorem which asserts that $k$ is conditionally negative definite if and only if $e^{-tk}$ is positive definite for all $t>0$.

Let $H$ be a Hilbert space, and let $a:X\to H$ be a map. Then the kernel $\phi:X\times X\to\C$ defined by
$$
\phi(x,y) = \la a(x),a(y)\ra
$$
is positive definite. Conversely, every positive definite kernel is of this form for some suitable Hilbert space $H$ and map $a$. On the other hand, the kernel $\psi:X\times X\to\C$ defined by
$$
\psi(x,y) = ||a(x) - a(y)||^2
$$
is conditionally negative definite, and every real-valued, conditionally negative definite kernel that vanishes on the diagonal $\{(x,x) \mid x\in X\}$ is of this form.

It is well-known that the set of positive definite kernels on $X$ is closed under pointwise products and pointwise convergence. Also, the set of conditionally negative definite kernels is closed under adding constants and under pointwise convergence. We refer to \cite[Chapter~3]{BCR} for details.

A kernel $k:X\times X\to\C$ is called a \emph{Schur multiplier} if for every operator $A = [a_{xy}]_{x,y\in X} \in B(\ell^2(X))$ the matrix $[k(x,y)a_{xy}]_{x,y\in X}$ represents an operator in $B(\ell^2(X))$, denoted $m_k(A)$. If $k$ is a Schur multiplier, it is a consequence of the closed graph theorem that $m_k$ defines a \emph{bounded} operator on $B(\ell^2(X))$. We define the \emph{Schur norm} $\|k\|_S$ to be $\|m_k\|$. The following characterization of Schur multipliers is well-known (see \cite[Appendix D]{BO}).

\begin{prop}\label{prop:Gilbert}
Let $k:X\times X\to\C$ be a kernel, and let $C\geq 0$ be given. The following are equivalent.
\begin{enumerate}
	\item The kernel $k$ is a Schur multiplier with $\|k\|_S \leq C$.
	\item There exist a Hilbert space $H$ and two bounded maps $a,b:X\to H$ such that
	$$
	k(x,y) = \la a(x),b(y)\ra, \qquad \text{for all }x,y\in X,
	$$
	and $\|a(x)\|\|b(y)\| \leq C$ for all $x,y\in X$.
\end{enumerate}
\end{prop}

Let $G$ be a discrete group, and let $\phi:G\to\C$ be a function. Let $\hat\phi:G\times G\to\C$ be defined by $\hat\phi(x,y) = \phi(y^{-1}x)$. All the terminology introduced above is inherited to functions $\phi:G\to\C$ by saying, for instance, that $\phi$ is positive definite if the kernel $\hat\phi$ is positive definite. The only exception is that a function $\phi:G\to\C$ is called a \emph{Herz-Schur multiplier} if $\hat\phi$ is a Schur multiplier.

All positive definite functions on $G$ are of the form $\phi(x) = \la \pi(x)\xi,\xi\ra$ for a unitary representation $\pi$ on some Hilbert space $H$ and a vector $\xi\in H$. If $\phi$ is real, then $\pi$ may be taken as an orthogonal representation on a real Hilbert space.

The set of Herz-Schur multipliers on $G$ is denoted $B_2(G)$. It is a Banach space, in fact a Banach algebra, when equipped with the norm $\|\phi\|_{B_2} = \|\hat\phi\|_S = \| m_{\hat\phi}\|$. The unit ball $B_2(G)_1$ is closed in the topology of pointwise convergence. It was proved in \cite{MR753889} that the space of Herz-Schur multipliers coincides isometrically with the space of completely bounded Fourier multipliers.

Another useful algebra of functions on $G$ is the {Fourier-Stieltjes algebra}, denoted $B(G)$. It may be defined as the linear span of the positive definite functions on $G$. It is isometrically isomorphic to the dual of the full group \Cs-algebra of $G$, i.e., $B(G) \simeq C^*(G)^*$. Since any positive definite function is a Herz-Schur multiplier, it follows that $B(G) \subseteq B_2(G)$. Equality holds, if and only if $G$ is amenable (see \cite{MR806070} or Proposition~\ref{lem:bozejko} below).

Given \Cs-algebras $A$ and $B$ and a linear map $\phi:A\to B$ we denote by $\phi^{(n)}$ the map $\phi^{(n)} = \phi\otimes \id_n : A\otimes M_n(\C) \to B\otimes M_n(\C)$, where $\id_n: M_n(\C)\to M_n(\C)$ is the identity. The map $\phi$ is called \emph{completely bounded}, if
$$
\|\phi\|_\cb = \sup_n \| \phi^{(n)} \| < \infty.
$$
We say that $\phi$ is \emph{completely positive}, if each $\phi^{(n)}$ is positive between the \Cs-algebras $M_n(A)$ and $M_n(B)$. We abbreviate unital, completely positive as u.c.p. It is well-known that bounded functionals $\phi:A\to\C$ are completely bounded with $\|\phi\|_\cb = \|\phi\|$. States on \Cs-algebras as well as $^*$-homomorphism are completely positive.

\section{Characterization of the weak Haagerup property with constant 1}\label{sec:wh-char}

The following theorem gives the promised alternative characterization of the weak Haagerup property with constant 1.

\begin{thm}\label{lem:WH}
Let $G$ be a countable, discrete group. The following are equivalent.
\begin{enumerate}
	\item  There is a sequence $(\phi_n)$ of functions vanishing at infinity such that $\phi_n\to 1$ pointwise and $\|\phi_n\|_{B_2} \leq 1$ for all $n$.
	\item There is $\phi:G \to \R$ such that $\phi$ is proper and $|| e^{-t\phi}||_{B_2} \leq 1$ for every $t>0$.
\end{enumerate}
\end{thm}

\begin{proof}
\mbox{}\newline
(2) $\implies$ (1): This is trivial: put $\phi_n = e^{-\phi/n}$.

\noindent
(1) $\implies$ (2): Choose an increasing, unbounded sequence $(\alpha_n)$ of positive real numbers and a decreasing sequence $(\epsilon_n)$ tending to zero such that $\sum_n \alpha_n\epsilon_n$ converges. We enumerate the elements in $G$ as $G = \{g_1,g_2,\ldots\}$. For each $n$ we may choose a function $\phi_n$ in $C_0(G)$ with $\|\phi_n\|_{B_2}\leq 1$ such that
$$
\max \{ |1- \phi_n(g_i) | \mid i=1,\ldots,n \} \leq \epsilon_n .
$$

We may replace $\phi_n$ by $|\phi_n|^2$ to ensure that $0 \leq \phi_n \leq 1$. Now, let $\phi: G\to\R_+$ be given by
$$
\phi(g) = \sum_{n=1}^\infty \alpha_n (1-\phi_n(g)).
$$
Note that this sum converges. We claim that $\phi$ is also proper. Let $R > 0$ be given, and choose $k$ such that $\alpha_k \geq 2R$. Since $\phi_k \in C_0(G)$, there is a finite set $F \subseteq G$ such that $|\phi_k(g)| < 1/2$ whenever $g\in G\setminus F$. Now if $\phi(g) \leq R$, then $\phi(g) \leq \alpha_k / 2$, and in particular $\alpha_k (1-\phi_k(g)) \leq \alpha_k/2$, which implies that $1-\phi_k(g) \leq 1/2$. Hence, we have argued that
$$
\{g\in G \mid \phi(g) \leq R \} \subseteq \{g\in G \mid 1-\phi_k(g) \leq 1/2 \}\subseteq F.
$$
This proves that $\phi$ is proper.

Now let $t>0$ be fixed. We need to show that $\|e^{-t\phi}\|_{B_2} \leq 1$. Define
$$
\psi_i = \sum_{n=1}^i \alpha_n (1-\phi_n).
$$
Since $\psi_i$ converges pointwise to $\phi$, it will suffice to prove that $\|e^{-t\psi_i}\|_{B_2} \leq 1$ eventually (as $i\to\infty$), because the unit ball of $B_2(G)$ is closed under pointwise limits. Observe that
$$
e^{-t\psi_i} = \prod_{n=1}^i e^{-t \alpha_n (1-\phi_n) },
$$
and so it suffices to show that $e^{-t \alpha_n (1-\phi_n) }$ belongs to the unit ball of $B_2(G)$ for each $n$. And this is clear, since
$$
\| e^{-t \alpha_n (1-\phi_n) } \|_{B_2} = e^{-t \alpha_n} \| e^{t \alpha_n \phi_n } \|_{B_2} \leq e^{-t \alpha_n}  e^{t \alpha_n \| \phi_n\|_{B_2} } \leq 1.
$$
\end{proof}

\section{Splitting a semigroup generator into positive and negative parts}\label{sec:A}

The key idea in the proof of Theorem~\ref{thm:A} is that a Schur multiplier is a corner in a positive definite matrix (Lemma~\ref{lem:Gilb}). Together with an ultraproduct argument this will give the proof of Theorem~\ref{thm:A}.

We consider the following as well-known.
\begin{lem}\label{lem:sup}
Let $\phi:X\times X\to\C$ be a kernel. Then
$$
||\phi||_S = \sup \{ || \phi|_{F\times F}||_S \mid F\subseteq X \text{ finite}\}.
$$
\end{lem}

\noindent
The following follows from Proposition~\ref{prop:Gilbert}.

\begin{lem}\label{lem:Gilb}
Let $a\in M_n(\C)$. The following are equivalent.
\begin{enumerate}
	\item $||a||_S \leq 1$.
	\item There exist $b,c\in M_n(\C)_+$ with $b_{ii} \leq 1$, $c_{ii} \leq 1$, $i=1,\ldots,n$ such that
	$$
	\begin{pmatrix}
	b & a \\
	a^* & c
	\end{pmatrix} \geq 0.
	$$
	\end{enumerate}
\end{lem}
\begin{proof}
Let $X = \{1,\ldots,n\}$ and consider $a \in M_n(\C)$ as a kernel $a:X\times X\to\C$.

Suppose first $\|a\|_S \leq 1$. By Proposition~\ref{prop:Gilbert} there is a Hilbert space $H$ and two maps $p,q:X\to H$ such that $a_{ij} = \la p(i),q(j)\ra$ and $\|p\|_\infty\|q\|_\infty \leq 1$ for all $i,j$. After replacing $p(i)$ and $q(j)$ by
$$
p'(i) = \frac{\sqrt{\|p\|_\infty\|q\|_\infty}}{\|p\|_\infty}\ p(i), \quad
q'(j) = \frac{\sqrt{\|p\|_\infty\|q\|_\infty}}{\|q\|_\infty}\ q(j)
$$
respectively, we may assume that $\|p\|_\infty\leq 1$ and $\|q\|_\infty\leq 1$. Let
$$
b_{ij} = \la p(i),p(j)\ra, \quad c_{ij} = \la q(i),q(j)\ra.
$$
Then $b$ and $c$ are positive matrices with diagonal below 1 and the matrix
$$
M=\begin{pmatrix}
	b & a \\
	a^* & c
	\end{pmatrix} =
	\bigg(
	\la r(i),r(j)\ra
	\bigg)_{i,j=1}^{2n}
$$
is positive where
$$
r(i) = \left\{
\begin{array}{ll}
p(i), & 1\leq i\leq n, \\
q(i-n), & n< i\leq 2n. \\
\end{array}
\right.
$$

\noindent
Conversely, suppose that
  $$
M=\begin{pmatrix}
	b & a \\
	a^* & c
	\end{pmatrix} \geq 0.
	$$
for some $b,c\in M_n(\C)_+$ with $b_{ii} \leq 1$, $c_{ii}\leq 1$. Then there is a Hilbert space $H$ and map $r : \{1,\ldots,2n\} \to H$ such that $M_{ij} = \la r(i),r(j)\ra$ for $i,j = 1,\ldots, 2n$. Put $p(i) = r(i)$ and $q(i) = r(i+n)$, $i=1,\ldots,n$. Then $a_{ij} = \la p(i),q(j)\ra$ and
$$
\|p(i)\|^2 = b_{ii} \leq 1, \quad \|q(j)\|^2 = c_{jj} \leq 1.
$$
It now follows from Proposition~\ref{prop:Gilbert} that $\|a\|_S \leq 1$.
\end{proof}

\noindent
Theorem~\ref{thm:A} is an immediate consequence of the following.

\begin{prop}\label{prop:gen}
Let $G$ be a countable, discrete group with a symmetric function $\phi:G\to\R$. The following are equivalent.
\begin{enumerate}
\item $\|e^{-t\phi}\|_{B_2}\leq 1$ for every $t> 0$.
\item There exist a real Hilbert space $H$ and maps $R,S:G\to H$ such that
$$
\phi(y^{-1}x) = ||R(x) - R(y)||^2 + ||S(x) + S(y)||^2 \quad\text{for all } x,y\in G.
$$
In particular, $||S(x)||^2$ is constant.
\end{enumerate}
\end{prop}
\begin{proof}

We will need to work with two disjoint copies of $G$, so let $\bar G$ denote another copy of $G$. We denote the elements of $\bar G$ by $\bar g$, when $g\in G$.

(2) $\implies$ (1): It suffices to prove the case when $t = 1$. After replacing the maps $R$ and $S$ by the maps $R',S':G\to H\oplus H$ given by $R'(x) = (R(x),0)$ and $S'(x) = (0,S(x))$, we may assume that $R$ and $S$ have orthogonal ranges. Then
$$
\phi(y^{-1}x) = ||R(x) + S(x) - ( R(y) - S(y) ) ||^2
$$
Let $P = R+S$ and $Q= R-S$. Define a map $T : G\sqcup\bar G \to H$ by
$$
T(x) = \begin{cases}
P(x), & x\in G,  \\
Q(x), & x\in \bar G.
\end{cases}
$$

Then the function $\rho(x,y) = ||T(x) - T(y)||^2$ is a conditionally negative definite kernel on the set $G\sqcup\bar G$, and by Schoenberg's Theorem the function $e^{-\rho}$ is positive definite, and we notice that $e^{-\rho}$ takes the value $1$ on the diagonal.

Given any finite subset $F = \{g_1,\ldots,g_n\}$ of $G$ we let $\bar F$ denote its copy inside $\bar G$. We see that the $2n\times 2n$ matrix $(e^{-\rho(x,y)})_{x,y\in F\sqcup\bar F}$ in $B(\ell^2(F \sqcup \bar F))$ is
$$
A = 
\left(\begin{array}{c|c}
& \\
e^{-||P(g_i)-P(g_j)||^2} & e^{-||P(g_i)-Q(g_j)||^2} \\
& \\
\hline
& \\
e^{-||Q(g_i)-P(g_j)||^2} & e^{-||Q(g_i)-Q(g_j)||^2} \\
& \\
\end{array}\right)
$$
Since $e^{-\rho}$ is positive definite, $A$ is positive. Now, Lemma~\ref{lem:Gilb} implies that the upper right block of $A$ has Schur norm at most $1$. And this precisely means that $|| e^{-\phi} |_F ||_S \leq 1$. An application of Lemma~\ref{lem:sup} now shows that $|| e^{-\phi}  ||_S \leq 1$.

(1) $\implies$ (2): We list the elements of $G$ as $G = \{g_1,g_2,\ldots\}$ and we let $G_n = \{g_1,\ldots,g_n\}$ when $n\in\N$. Since $||e^{-\phi/n}||_{B_2} \leq 1$ by assumption, we invoke Lemma~\ref{lem:Gilb} to get matrices $b_n,c_n\in M_n(\C)_+$ with diagonal entries at most one, and so that
$$
A_n = \left(\begin{array}{c|c}
b_n & e^{-\phi/n} \\ \hline
e^{-\phi/n} & c_n
\end{array}\right) \geq 0.
$$
Here $e^{-\phi/n}$ denotes the $n\times n$ matrix whose $(i,j)$ entry is $e^{-\phi(g_j^{-1}g_i)/n}$. After adding the appropriate diagonal matrix we may assume that the diagonal entries of $b_n$ and $c_n$ are $1$, and $A_n$ is still positive.

Let $k_n: (G_n \sqcup \bar G_n)^2 \to \C$ be the kernel that represents $A_n$ in the sense that
\begin{align*}
k_n(g_i,g_j) &= (b_n)_{i,j}, &
k_n(\bar g_i,\bar g_j) &= (c_n)_{i,j}, \\
k_n(g_i,\bar g_j) &= e^{-\phi(g_j^{-1}g_i)}, &
k_n(\bar g_i,g_j) &= e^{-\phi(g_j^{-1}g_i)},
\end{align*}
Since $A_n$ is positive, $k_n$ is a positive definite kernel. We define $k_n$ to be zero outside $(G_n \sqcup\bar G_n)^2$, which gives us a positive definite kernel on $G\sqcup \bar G$. Then the function $n ( 1 - k_n)$ is a conditionally negative definite kernel with zero in the diagonal, and hence there is a map $T_n: G\sqcup \bar G \to H_n$ such that
$$
||T_n(x) - T_n(y)||^2 = n( 1 - k_n(x,y)), \quad x,y\in G\sqcup \bar G
$$
for some real Hilbert space $H_n$. We may assume that $T_n(\bar e) = 0$.

Now, as we let $n$ vary over $\N$ we obtain a sequence of maps $(T_n)_{n\geq 1}$. Because $\lim_{t\to 0} (1- e^{-ta})/t = a$, we see that for $(x,\bar y)\in G_N \times \bar G_N$ and $n\geq N$
$$
||T_n(x) - T_n(\bar y)||^2 = n( 1 - k_n(x,\bar y)) = n(1- e^{-\phi(y^{-1}x)/n}) \to \phi(y^{-1}x) \text{ as } n\to\infty.
$$
Since $T_n(\bar e) = 0$, this shows in particular that $(||T_n(x)||)_{n\geq 1}$ is a bounded sequence for each $x\in G$ and hence also for each $x\in\bar G$.

Consider the ultraproduct of the Hilbert spaces $H_n$ with respect to some free ultrafilter $\omega$. We denote this space by $H$. Let $T(x)$ denote the vector corresponding to the sequence $(T_n(x))_{n\geq 1}$, i.e., the equivalence class of that sequence. Then
\begin{align}\label{eq:PQ}
\phi(y^{-1}x) = ||T(x) - T(\bar y)||^2 \quad\text{for every } (x,\bar y)\in G\times\bar G.
\end{align}
Let $P = T|_G$ and let $Q$ be defined on $G$ by $Q(x) = T(\bar x)$. We  think of $Q$ as the restriction of $T$ to $\bar G$ but defined on $G$. Then Equation~\eqref{eq:PQ} translates to
$$
\phi(y^{-1}x) = ||P(x) - Q(y)||^2 \quad\text{for every } x,y\in G.
$$
Let $R = (P+Q)/2$ and $S = (P-Q)/2$. The rest of the proof is simply to apply the parallelogram law. We have
$$
\frac12\left( ||P(x) - Q(y)||^2 + ||P(y) - Q(x)||^2 \right) = || S(x) + S(y) ||^2
+ || R(x) - R(y) ||^2.
$$
Since $\phi$ is symmetric, the left-hand side equals $\phi(y^{-1}x)$, and the proof is complete.

\end{proof}

\section{The amenable case}\label{sec:amenable}

In this section we prove Theorem~\ref{thm:amenable-case}. Theorem \ref{prop:amenable-case3} and Theorem~\ref{prop:non-amenable} combine to give Theorem~\ref{thm:amenable-case}.

We will need a few characterizations of amenability. The following theorem is well-known (for a proof, see \cite[Theorem 2.6.8]{BO}).

\begin{thm}\label{thm:amenable}
Let $G$ be a discrete group. The following are equivalent.
\begin{enumerate}
	\item $G$ is amenable, i.e., there is a left-invariant, finitely additive probability measure defined on all subsets of $G$.
	\item There is a net of finitely supported, positive definite functions on $G$ converging pointwise to $1$.
	\item For any finite, symmetric subset $E\subseteq G$ we have $\|\lambda(1_E) \| = |E|$. Here $\lambda$ denotes the left regular representation, and $1_E$ denotes the characteristic function of the subset $E$.
\end{enumerate}
\end{thm}

\begin{cor}\label{cor:non-amenable}
If $G$ is discrete and non-amenable, then for each $\epsilon > 0$ there exists a positive, finitely supported, symmetric function $g\in C_c(G)$ such that
$$
\|\lambda(g)\| < \epsilon \|g\|_1,
$$
where $\lambda$ denotes the left regular representation.
\end{cor}
\begin{proof}
If $G$ is non-amenable, there is a finite, symmetric set $S\subseteq G$ such that $\|\lambda(1_S)\| < |S|$. Let $g =
1_S * \cdots * 1_S
$ be the $n$-fold convolution of $1_S$ with itself, where $n$ is to be determined later. Then $g$ is positive, finitely supported and symmetric. Observe that
$$
\| g\|_1 = |S|^n.
$$
Now, given any $0< \epsilon < 1$, choose $n$ so large that $\frac{\|\lambda(1_S)\|}{|S|} < \sqrt[n]\epsilon < 1$. Then
$$
\|\lambda(g)\| \leq \| \lambda(1_S) \|^n < \epsilon |S|^n = \epsilon \| g\|_1.
$$
\end{proof}

The following theorem proves one direction in Theorem~\ref{thm:amenable-case}.

\begin{thm}\label{prop:amenable-case3}
Let $G$ be a countable, discrete amenable group with a symmetric function $\phi:G\to\R$. If $\|e^{-t\phi}\|_\HS \leq 1$ for every $t>0$, then $\phi$ splits as
$$
\phi(x) = \psi(x) + ||\xi||^2 + \la\pi(x)\xi,\xi\ra  \qquad (x\in G),
$$
where
\begin{itemize}
\renewcommand{\labelitemi}{$\cdot$}
	\item $\psi$ is a conditionally negative definite function on $G$ with $\psi(e) = 0$,
	\item $\pi$ is an orthogonal representation of $G$ on some real Hilbert space $H$,
	\item and $\xi$ is a vector in $H$.
\end{itemize}
\end{thm}

\begin{proof}
The idea of the proof is to use the characterization given in Proposition~\ref{prop:gen} and then average the two parts of $\phi$ by using an invariant mean on $G$.

Suppose we are given a function $\phi$ as in the statement of the proposition. By Proposition~\ref{prop:gen} we may write $\phi$ in the form
$$
\phi(y^{-1}x) = ||R(x) - R(y)||^2 + ||S(x) + S(y)||^2 \quad\text{for all } x,y\in G,
$$
where $R,S$ are maps from $G$ with values in some real Hilbert space $H$. We define kernels $\phi_1$ and $\phi_2$ on $G$ by
$$
\phi_1(x,y) = ||R(x) - R(y)||^2, \quad \phi_2(x,y) = ||S(x) + S(y)||^2 \quad\text{for all } x,y\in G.
$$
Then $\phi(y^{-1}x) = \phi_1(x,y) + \phi_2(x,y)$. Note that $\phi_2$ is a bounded function, since $||2S(x)||^2 = \phi(e)$ for every $x\in G$. In general, $\phi_1$ is not bounded, but it is bounded on the diagonals, i.e., for each $x,y\in G$ the function $z\mapsto \phi_1(zx,zy)$ is bounded. To see this, simply observe that
$$
\phi_1(zx,zy) = \phi((zy)^{-1}(zx)) - \phi_2(zx,zy) = \phi(y^{-1}x) - \phi_2(zx,zy).
$$
Since $\phi_2$ is bounded, it follows that $\phi_1$ is bounded on diagonals.

As we assumed $G$ to be amenable, there is a left-invariant mean $\mu$ on $G$. Let
$$
\chi_i(x,y) = \int_G \phi_i(zx,zy) \ d\mu(z) \qquad x,y\in G,\ i=1,2.
$$
The left-invariance of $\mu$ implies that $\chi_i(wx,wy) = \chi_i(x,y)$ for every $x,y,w\in G$, so $\chi_i$ induces a function $\bar\phi_i$ defined on $G$ by
$$
\bar\phi_i(y^{-1}x) = \chi_i(x,y).
$$
An easy computation will show that $\phi = \bar\phi_1 + \bar\phi_2$.

Since $\phi_1$ is a conditionally negative definite kernel on $G$, it follows that $\chi_1$ is conditionally negative definite. So $\bar\phi_1$ is conditionally negative definite. The same argument shows that $\bar\phi_2$ is positive definite, because $\phi_2$ is. More precisely we have
$$
\bar\phi_2(y^{-1}x) = \int_G ||S(zx) + S(zy)||^2 \ d\mu(z) = \frac{\phi(e)}{2} +  2 \int_G \la S(zx), S(zy)\ra \ d\mu(z),
$$
where each function $(x,y)\mapsto \la S(zx),S(zy)\ra$ is a positive definite kernel. So the function on $G$ given by
$$
y^{-1}x \mapsto \int_G \la S(zx), S(zy)\ra \ d\mu(z)
$$
is positive definite, and so it has the form
$$
g \mapsto \la \pi(g)\xi',\xi'\ra
$$
for some orthogonal representation $\pi$. Since $\bar\phi_1(e) = 0$, we must have
$$
\phi(e) = \bar\phi_2(e) = \frac{\phi(e)}{2} + 2 ||\xi'||^2,
$$
and so $\frac{\phi(e)}{2} = 2 ||\xi'||^2$. The proof is now complete if we let $\psi = \bar\phi_1$ and $\xi = \sqrt2\xi'$.
\end{proof}

We now turn to prove that the amenability assumption is essential in the theorem above. This will be accomplished in Theorem~\ref{prop:non-amenable}.

In \cite{MR806070} Bo\.{z}ejko proved that a countable, discrete group $G$ is amenable if and only if its Fourier-Stieltjes algebra $B(G)$ (the linear span of positive definite functions) coincides with the Herz-Schur multiplier algebra $B_2(G)$. In Proposition~\ref{lem:bozejko} we will strengthen this result slightly to fit our needs. Our proof of Proposition~\ref{lem:bozejko} is merely an adaption of Bo\.{z}ejko's proof.

In the following we will introduce the \emph{Littlewood kernels} and \emph{Littlewood functions}. Let $X$ be a non-empty set. A bounded operator $T : \ell^1(X) \to \ell^2(X)$ is identified with its matrix $T = [T_{xy}]$ given by $T_{xy} = \la T\delta_y,\delta_x \ra$. We also identify the matrix with the corresponding kernel $t$ on $X$ given by $t(x,y) = T_{xy}$. Similarly, the Banach space adjoint $T^* : \ell^2(X)^* \to \ell^1(X)^*$ has matrix $T^*_{xy} = \bar{\la T\delta_x,\delta_y\ra}$ and may be identified with a kernel on $X$.

We shall identify $\ell^1(X)^* = \ell^\infty(X)$ and $\ell^2(X)^* = \ell^2(X)$. It is known that every bounded operator $\ell^2(X) \to \ell^\infty(X)$ arises as the adjoint of a (unique) bounded operator $\ell^1(X) \to \ell^2(X)$. We note that a kernel $b:X\times X\to\C$ is the matrix of a bounded operator $\ell^1(X)\to\ell^2(X)$ if and only if
$$
\|b\|_{\ell^1\to\ell^2}^2 = \sup_{y\in X} \sum_{x\in X} | b_{xy} |^2
$$
is finite. In the same way, $c:X\times X\to\C$ is the matrix of a bounded operator $\ell^2(X)\to\ell^\infty(X)$ if and only if
$$
\|c\|_{\ell^2\to\ell^\infty}^2 = \sup_{x\in X} \sum_{y\in X} | c_{xy} |^2
$$
is finite.

We define the \emph{Littlewood kernels} on $X$ to be
$$
t_2(X) = \{ b + c \mid b\in B(\ell^1(X),\ell^2(X)), c\in B(\ell^2(X),\ell^\infty(X)) \}.
$$
The space $t_2(X)$ is naturally equipped with the (complete) norm
$$
\| a \|_L = \inf \{ \max(\|b\|_{\ell^1\to\ell^2} , \|c\|_{\ell^2\to\ell^\infty} ) \mid a = b+c \}.
$$
The following characterization of Littlewood kernels is due to Varopoulos and is a special case of \cite[Lemma 5.1]{MR0355642}. For completeness, we include a proof of our special case.
\begin{lem}
Let $X$ be a countable set, and let $a:X\times X\to\C$ be a kernel. Then $a$ is a Littlewood kernel if and only if the norm
$$
\| a \|_{t_2} = \sup \left\{ \frac{1}{|F_1|} \sum_{\substack{i\in F_1 \\ j\in F_2}} |a_{ij}|^2 \middle| F_1,F_2\subseteq X \text{ finite}, |F_1| = |F_2| \right\}^{1/2}
$$
is finite.
The norms $\|\ \|_{t_2}$ and $\|\ \|_L$ are equivalent.
\end{lem}
\noindent
It is implicit in the statement that $\| \ \|_{t_2}$ in fact defines a norm on $t_2(X)$. This is not hard to check, and moreover $t_2(X)$ is a Banach space with this norm. The lemma is also true when $X$ is uncountable, but we have no need for this generality.

\begin{proof}
Suppose first that $a$ is a Littlewood kernel, and write $a = b + c$ as in the definition. Given finite subsets $F_1,F_2 \subseteq X$ of the same size we have
$$
\frac{1}{|F_1|} \sum_{\substack{i\in F_1 \\ j\in F_2}} |a_{ij}|^2 \leq \frac{2}{|F_1|} \sum_{\substack{i\in F_1 \\ j\in F_2}} |b_{ij}|^2 + |c_{ij}|^2 \leq 4 \max(\| b\|^2 , \| c\|^2),
$$
and so
$$
\| a\|_{t_2} \leq 2 \max(\|b\| , \|c\|) < \infty.
$$
We have actually shown that $\|a\|_{t_2} \leq 2 \|a\|_L$.

Suppose conversely that $a$ is a kernel such that $C = \|a\|_{t_2}$ is finite. We will show that $a$ is a Littlewood kernel of the form $b + c$, where $b$ and $c$ have disjoint supports, and $\|b\| \leq C$ and $\|c\| \leq C$. We finish the proof of the lemma first in the case where $X$ is finite and proceed by induction on $|X|$. The case $|X| = 1$ is trivial. Assume then $n = |X| \geq 2$ and write
$$
a = \begin{pmatrix}
a_{11} & \cdots & a_{1n}   \\
\vdots & \ddots & \vdots \\
a_{n1} & \cdots & a_{nn}
\end{pmatrix}.
$$
Choose an index $i$ such that $\sum_j |a_{ij}|^2$ is as small as possible. In particular, our assumption implies that $\sum_j |a_{ij}|^2 \leq C$. Similarly, choose an index $j$ such that $\sum_i |a_{ij}|^2$ is as small as possible. Consider then the submatrix $a'$ of $a$ with $i$'th row and $j$'th column removed. To simplify the notation we assume that $i = j = 1$. Then
$$
a' = 
\begin{pmatrix}
a_{22} & \cdots & a_{2n}   \\
\vdots & \ddots & \vdots \\
a_{n2} & \cdots & a_{nn}
\end{pmatrix}.
$$
By our induction hypothesis $a'$ is a Littlewood kernel with $\|a'\|_L \leq C$, and so we may write $a' = b + c$, that is
$$
\begin{pmatrix}
a_{22} & \cdots & a_{2n}   \\
\vdots & \ddots & \vdots \\
a_{n2} & \cdots & a_{nn}
\end{pmatrix}
=
\begin{pmatrix}
b_{22} & \cdots & b_{2n}   \\
\vdots & \ddots & \vdots \\
b_{n2} & \cdots & b_{nn}
\end{pmatrix}
+
\begin{pmatrix}
c_{22} & \cdots & c_{2n}   \\
\vdots & \ddots & \vdots \\
c_{n2} & \cdots & c_{nn}
\end{pmatrix},
$$
where $b$ and $c$ have disjoint supports and $\max(\|b\|,\|c\|)\leq C$. We then obtain the desired splitting for $a$ by putting the removed rows back (we do not care whether $a_{ij} = a_{11}$ is put in the first or second matrix, so simply put it in the first),
$$
\left(\begin{array}{c|ccc}
a_{11} & a_{12} & \cdots & a_{1n} \\ \hline
a_{21} & a_{22} & \cdots & a_{2n} \\
\vdots & \vdots & \ddots & \vdots \\
a_{n1} & a_{n2} & \cdots & a_{nn}
\end{array}\right)
=
\left(\begin{array}{c|ccc}
a_{11} & a_{12} & \cdots & a_{1n} \\ \hline
0 & b_{22} & \cdots & b_{2n} \\
\vdots & \vdots & \ddots & \vdots \\
0 & b_{n2} & \cdots & b_{nn}
\end{array}\right)
+
\left(\begin{array}{c|ccc}
0 & 0 & \cdots & 0 \\ \hline
a_{21} & c_{22} & \cdots & c_{2n} \\
\vdots & \vdots & \ddots & \vdots \\
a_{n1} & c_{n2} & \cdots & c_{nn}
\end{array}\right).
$$
This completes the induction step.

We now turn to the general case, where $X$ is countably infinite. We may assume $X = \N$. Let $\omega$ be a free ultrafilter on $\N$. For each $k\in\N$ we let $a^{(k)}$ be the restriction of $a$ to $\{1,\ldots,k\}^2$, and choose a splitting $a^{(k)} = b^{(k)} + c^{(k)}$. For each $i,j\in\N$ and $k\geq i,j$ we have $|b^{(k)}_{ij}|^2 \leq C$, so each sequence $(b^{(k)}_{ij})_k$ is bounded. Similarly, $(c^{(k)}_{ij})_k$ is bounded. Let
$$
b_{ij} = \lim_{k\to \omega} b^{(k)}_{ij}, \qquad c_{ij} = \lim_{k\to \omega} c^{(k)}_{ij}.
$$
Since $a_{ij} = b^{(k)}_{ij} + c^{(k)}_{ij}$ for every $k \geq i,j$, it follows that $a_{ij} = b_{ij} + c_{ij}$. Also, since $b^{(k)}_{ij} \in \{a_{ij} , 0 \}$ for every $k$, we must have $b_{ij}\in\{ a_{ij}, 0\}$. Similarly with $c_{ij}$. This shows that $b$ and $c$ have disjoint supports. The sum conditions
$$
\sup_j \sum_i |b_{ij}|^2 \leq C, \qquad \sup_i \sum_j |c_{ij}|^2 \leq C
$$
are also satisfied. In particular we have $\|a\|_L \leq \|a\|_{t_2}$.
\end{proof}

If $X = G$ is a group, and $a:G\to\C$ is a function, we say that $a$ is a \emph{Littlewood function}, if $\hat a(x,y) = a(y^{-1}x)$ is a Littlewood kernel. We denote the set of Littlewood functions on $G$ by $T_2(G)$ and equip it with the norm $\|a\|_{T_2} = \|\hat a\|_{t_2}$. It is easy to see that $\|a\|_{T_2} \leq \|a\|_{\ell^2}$, so $\ell^2(G) \subseteq T_2(G)$.

Let $M(\ell^\infty(G),B_2(G)) = M(\ell^\infty,B_2)$ be the set of functions $a:G\to\C$ such that the pointwise product $a\cdot f$ is a Herz-Schur multiplier for every $f\in\ell^\infty(G)$. It is a Banach space when equipped with the norm
$$
\|g\|_{M(\ell^\infty,B_2)} = \sup \{ \|a\cdot f\|_{B_2} \mid \|f\|_\infty \leq 1\}.
$$

\begin{lem}\label{lem:littlewood-multiplier}
The following inclusion holds.
$$
T_2(G) \subseteq M(\ell^\infty(G),B_2(G)).
$$
\end{lem}
\begin{proof}
Note first that $T_2(G)\subseteq B_2(G)$, since if $a\in T_2(G)$ is given, and $\hat a = b + c$ is a splitting as in the definition of Littlewood kernels, then
$$
a(y^{-1}x) = \la b\delta_y,\delta_x\ra + \bar{\la c\delta_x,\delta_y\ra}.
$$
Now use Proposition~\ref{prop:Gilbert}.

Secondly, it is easy to see that $t_2(X) \cdot \ell^\infty(X\times X) \subseteq t_2(X)$, and we conclude that
$$
T_2(G)\cdot \ell^\infty(G) \subseteq T_2(G) \subseteq B_2(G).
$$
\end{proof}

In the proof of Proposition~\ref{lem:bozejko} we will need the notion of a cotype 2 Banach space. A Banach space $X$ is of \emph{cotype 2} if there is a constant $C>0$ such that for any finite subset $\{x_1,\ldots,x_n\}$ of $X$ we have
$$
C\left( \sum_{k=1}^n \| x_k \|^2 \right)^{1/2} \leq \int_0^1 \left\| \sum_{k=1}^n r_k(t) x_k \right\| \ dt.
$$
Here $r_n$ are the Rademacher functions on $[0,1]$. It is well-known that $L^p$-spaces are of cotype 2 when $1\leq p \leq 2$. Also, the dual of a \Cs-algebra is of cotype 2 (see \cite{MR0355667}). (See also \cite{MR512252} for a simple proof of this fact.)

Whenever $A$ is a set of functions $G\to\C$ defined on a group $G$, we denote by $A_\sym$ the symmetric functions in $A$, i.e., $A_\sym = \{\phi \in A \mid \phi(x) = \phi(x^{-1}) \text{ for all } x\in G \}$.

\begin{prop}\label{lem:bozejko}
Let $G$ be a discrete group. The following are equivalent.
\begin{enumerate}
	\item $G$ is amenable.
	\item $B_2(G) = B(G)$.
	\item $B_2(G)_\sym = B(G)_\sym$.
\end{enumerate}
\end{prop}
\begin{proof}
For the implication (1)$\implies$(2) we refer to Theorem 1 in \cite{MR0293019}. The implication (2)$\implies$(3) is trivial. So we prove (3)$\implies$(1), and we do this by adapting Bo\.{z}ejko's proof of (2)$\implies$(1).

Since $B(G)$ may be identified with the dual of the full group \Cs-algebra of $G$, it is of cotype 2. Being of cotype 2 obviously passes to (closed) subspaces, so $B(G)_\sym$ is of cotype 2. By assumption $B_2(G)_\sym = B(G)_\sym$, and because the two spaces have equivalent norms, $B_2(G)_\sym$ is also of cotype 2.

Now we show that
$$
M(\ell^\infty(G),B_2(G))_\sym \subseteq \ell^2(G).
$$
Suppose $g\in M(\ell^\infty(G),B_2(G))_\sym$ and write $g$ in the form
$$
g = \sum_{n=1}^\infty a_n (\delta_{x_n} + \delta_{x_n^{-1}})
$$
with no repetitions among the sets $\{x_n,x_n^{-1}\}_{n=1}^\infty$. For each $t\in[0,1]$ and $N\in\N$ the function
$$
g_{t,N} = \sum_{n=1}^N a_n r_n(t) (\delta_{x_n} + \delta_{x_n^{-1}})
$$
lies in $B_2(G)_\sym$ and $\|g_{t,N} \|_{B_2} \leq \|g\|_{M(\ell^\infty,B_2)}$. Using that $B_2(G)_\sym$ has cotype 2 we get
$$
C \left( \sum_{n=1}^N |a_n|^2 \right)^{1/2} \leq \int_0^1 \left\| \sum_{n=1}^N r_n(t) a_n (\delta_{x_n} + \delta_{x_n^{-1}}) \right\|_{B_2}  dt \leq \| g\|_{M(\ell^\infty,B_2)}
$$
for any $N\in\N$, so $g\in \ell^2(G)$.

Now, consider the set $T^2(G)$ of Littlewood functions. As noted in Lemma~\ref{lem:littlewood-multiplier},
$$
T_2(G) \subseteq M(\ell^\infty(G), B_2(G)),
$$
so it follows that $T_2(G)_\sym \subseteq \ell^2(G)$. Conversely, the inclusion $\ell^2(G) \subseteq T_2(G)$ is trivial, so we must have $T_2(G)_\sym = \ell^2(G)_\sym$.

Let $\tilde f(x) = \bar{f(x^{-1})}$. It is easy to check that
$$
\la \lambda(f) x,y \ra = \la f, y * \tilde x \ra \qquad \text{for all } f,x,y\in \C[G],
$$
Hence for any symmetric $f\in \C[G]$ we have
\begin{align*}
\| f \|_{T_2}^2 &= \sup_{|F_1| = |F_2| < \infty} \left\{ \frac{1}{|F_1|} \la |f|^2 , \chi_{F_1} * \tilde{\chi}_{F_2} \ra \right\} \\
 &= \sup_{|F_1| = |F_2| < \infty} \left\{ \frac{1}{|F_1|} \la \lambda(|f|^2) \chi_{F_2}, \chi_{F_1}\ra \right\} \leq \| \lambda(|f|^2) \|.
\end{align*}
Since $T_2(G)_\sym = \ell^2(G)_\sym$, and these spaces have equivalent norms, we get
$$
\| |f| \|_{\ell^2}^2 \leq C' \| \lambda(|f|^2) \|  \qquad \text{for all } f\in \C[G]_\sym
$$
for some constant $C'$. This implies that
$$
\| g\|_{\ell^1} \leq C'' \| \lambda(g) \|
$$
for any positive, symmetric function $g\in \C[G]$ and some constant $C''$. Corollary~\ref{cor:non-amenable} yields that $G$ must be amenable.
\end{proof}

\begin{lem}
Let $G$ be a group, and $\psi:G\to\R$ a conditionally negative definite function. If $\psi$ is bounded, then $\psi = c - \phi$ for some constant $c\in\R$ and some positive definite function $\phi:G\to\R$.
\end{lem}
\begin{proof}
Without loss of generality we may assume $\psi(e) = 0$. It is then well-known that $\psi$ has the form $\psi(y^{-1}x) = \| \sigma(x) - \sigma(y) \|^2$ for some 1-cocycle $\sigma:G\to H$ with coefficients in an orthogonal representation $\pi:G\to O(H)$, where $H$ is a real Hilbert space. Since $\psi$ is bounded, so is $\sigma$. Any bounded 1-cocycle is a 1-coboundary (see \cite[Lemma D.10]{BO}). Hence there is $\xi\in H$ such that $\sigma(x) = \xi - \pi(x)\xi$ for every $x\in G$. Then
$$
\psi(y^{-1}x) = \| \sigma(x) - \sigma(y) \|^2 = 2||\xi||^2 - 2 \la \pi(y^{-1}x)\xi,\xi \ra.
$$
Now, put $c = 2\|\xi\|^2$ and $\phi(x) = 2\la \pi(x)\xi,\xi\ra$.
\end{proof}

We are now ready to prove the other direction of Theorem~\ref{thm:amenable-case}.

\begin{thm}\label{prop:non-amenable}
Let $G$ be a countable, discrete group. Suppose every symmetric function $\phi:G\to\R$ such that $\| e^{-t\phi}\|_\HS \leq 1$ for every $t>0$ splits as
$$
\phi(x) = \psi(x) + ||\xi||^2 + \la\pi(x)\xi,\xi\ra  \qquad (x\in G),
$$
where 
\begin{itemize}
\renewcommand{\labelitemi}{$\cdot$}
	\item $\psi$ is a conditionally negative definite function on $G$,
	\item $\pi$ is an orthogonal representation of $G$ on some real Hilbert space $H$,
	\item and $\xi$ is a vector in $H$.
\end{itemize}
Then $G$ is amenable.
\end{thm}

\begin{proof}

It is always the case that $B(G) \subseteq B_2(G)$. Suppose $\rho\in B_2(G)$ is real, symmetric, with $\|\rho\|_\HS = 1$. If we put $\phi = 1 - \rho$, then
$$
\| e^{-t\phi} \|_\HS = e^{-t} \| e^{t\rho} \|_\HS \leq e^{-t} e^{t\| \rho \|_\HS } = 1.
$$
By our assumption we have a splitting
$$
\phi(x) = \psi(x) + ||\xi||^2 + \la\pi(x)\xi,\xi\ra  \qquad (x\in G).
$$
Obviously, $\rho$ is bounded, and it follows that $\psi$ is bounded. By the previous lemma there is some positive definite function $\phi'$ on $G$ and a constant $c\in\R$ such that $\psi = c - \phi'$. Hence $\psi\in B(G)$. From this we get that $\phi \in B(G)$, so $\rho \in B(G)$.

It now follows that $B_2(G)_\sym \subseteq B(G)$, so $B_2(G)_\sym = B(G)_\sym$. From Proposition~\ref{lem:bozejko} we conclude that $G$ is amenable.
\end{proof}

\section{Radial semigroups of Herz-Schur multipliers on \texorpdfstring{$\F_\infty$}{F-infty}}\label{sec:infinite}
We now change the focus of Problem~\ref{prob:2} to the particular case where the group in question is a free group. We briefly recall Problem~\ref{prob:2}. Suppose $\phi:G\to\R$ is symmetric and $\|e^{-t\phi}\|_{B_2} \leq 1$ for every $t>0$. Is it then possible to find a conditionally negative definite function $\psi:G\to\R$ such that $\phi\leq\psi$. In the case of radial functions on free groups we provide a positive solution to the problem (see Theorem~\ref{thm:main-finite}).

Let $\N_0 = \{0,1,2,\ldots\}$, and let $\sigma:\N_0^2 \to \N_0^2$ denote the \emph{shift map} given as $\sigma(m,n) = (m+1,n+1)$. Let $\F_n$ denote the free group on $n$ generators, where $2\leq n\leq \infty$. We use $|x|$ to denote the word length of $x\in\F_n$.

\begin{defi}\label{defi:radial}
A function $\phi:\F_n \to\C$ is called \emph{radial} if there is a (necessarily unique) function $\dot\phi:\N_0\to \C$ such that $\phi(x) = \dot\phi(|x|)$ for all $x\in\F_n$, i.e., if the value $\phi(x)$ only depends on $|x|$.

A function $\phi:\N_0\times\N_0\to\C$ is called a \emph{Hankel function} if the value $\phi(m,n)$ only depends on $m+n$.

Given a radial function $\phi$ on $\F_n$, we let $\tilde\phi$ be the kernel on $\N_0$ defined by $\tilde\phi(m,n) = \dot\phi(m+n)$. Note that $\tilde\phi$ is a Hankel function.

\end{defi}

Actually, the free groups will not enter the picture before Theorem~\ref{thm:HSS}. Until then we will simply study properties of kernels on $\N_0$.

\subsection{Functionals on the Toeplitz algebra}

Let $S$ be the unilateral shift operator on $\ell^2(\N_0)$. The \Cs-algebra $C^*(S)$ generated by $S$ is the Toeplitz algebra. Since $S^*S = I$, the set
\begin{align}\label{eq:D}
D = \mathrm{span} \{S^k(S^*)^l \mid k,l\in\N_0\}
\end{align}
is a $^*$-algebra, and its closure is $C^*(S)$. The Toeplitz algebra fits in the exact sequence
\begin{align}\label{eq:extension}
\xymatrix{
0 \ar[r] & \mc K \ar[r] & C^*(S) \ar[r]^\pi & C(\mb T) \ar[r] & 0,
}
\end{align}
where $\mc K$ denotes the \Cs-algebra of compact operators (on $\ell^2(\N_0)$), $C(\mb T)$ is the \Cs-algebra of continuous functions on the unit circle $\mb T$, and $\pi$ is the quotient map that maps $S$ to the generating unitary $\id_{\mb T}$.

When $\phi:\N_0^2\to\C$ is a kernel we let $\omega_\phi$ denote the linear functional defined on $D$ by
\begin{align}\label{eq:omega}
\omega_\phi(S^m(S^*)^n) = \phi(m,n).
\end{align}
It may or may not happen that $\omega_\phi$ extends to a bounded functional on $C^*(S)$. If it does, we also denote the extension by $\omega_\phi$. Along the same lines we consider the linear map $M_\phi$ defined on $D$ by
$$
M_\phi(S^m(S^*)^n) = \phi(m,n)S^m(S^*)^n,
$$
and if it extends to a bounded linear map on $C^*(S)$, we also denote the extension by $M_\phi$. We call it the \emph{multiplier} of $\phi$.

\begin{rem}\label{rem:ucp}
Consider the \Cs-algebra $C^*(S\tensor S)$ generated by the operator $S\tensor S$ inside $B(\ell^2(\N_0)\tensor\ell^2(\N_0))$. Since the operator $S\tensor S$ is a proper isometry, it follows from Coburn's Theorem (see \cite[Theorem 3.5.18]{Murphy}) that there is a $^*$-isomorphism $\alpha:C^*(S)\to C^*(S\tensor S)$ such that $\alpha(S) = S\tensor S$. Let $\pi$ be the quotient map $C^*(S)\to C(\mb T)$ from before and let $\ev_1:C(\mb T)\to\C$ be evaluation at $1\in\mb T$. Then we note that $\omega_\phi = \ev_1\circ\pi\circ M_\phi$, while $M_\phi = (\id_{C^*(S)}\tensor \omega_\phi)\circ \alpha$, where we have identified $C^*(S)\tensor\C$ with $C^*(S)$.

It follows from the mentioned relation between $\omega_\phi$ and $M_\phi$ that $\omega_\phi$ extends to $C^*(S)$ if and only if $M_\phi$ extends to $C^*(S)$. If this is the case, then $M_\phi$ is even completely bounded, since bounded functionals and $^*$-homomorphisms are always completely bounded. Similarly, $\omega_\phi$ is positive if and only if $M_\phi$ is positive if and only if $M_\phi$ is \emph{completely} positive. Finally, $||M_\phi|| = ||\omega_\phi||$.
\end{rem}

\subsection{Positive and conditionally negative functions}
The following proposition characterizes the functions $\phi$ that induce states on the Toeplitz algebra.

\begin{prop}\label{prop:thm1}
Let $\phi:\N_0\times\N_0\to\C$ be a function. The following are equivalent.
\begin{enumerate}
	\item The functional $\omega_\phi$ extends to a state on the Toeplitz algebra $C^*(S)$.
	\item The multiplier $M_\phi$ extends to a u.c.p. map on the Toeplitz algebra $C^*(S)$.
	\item There exist a positive trace class operator $h = [h(i,j)]_{i,j=0}^\infty$ on $\ell^2(\N_0)$ and a positive definite function $\phi_0:\Z\to\C$ such that 
	$$
	\phi(k,l) = \sum_{i=0}^\infty h(k+i,l+i) + \phi_0(k-l)
	$$
	for all $(k,l)\in\N_0\times\N_0$ and $\phi(0,0) = 1$.
	\item $\phi$ is a positive definite kernel with $\phi(0,0) = 1$ and $\phi - \phi\circ\sigma$ is a positive definite kernel.
\end{enumerate}
\end{prop}

\begin{proof}
The equivalence (1)$\iff$(2) follows from Remark~\ref{rem:ucp}.

The rest of the proof goes as follows: (1)$\implies$(4)$\implies$(3)$\implies$(1).

(1)$\implies$(4): Given complex numbers $c_0,\ldots,c_n$, we see that
$$
\omega_\phi\left( \left(\sum_{k=0}^n c_k S^k \right) \left(\sum_{l=0}^n c_l S^l \right)^* \right) = \sum_{k,l=0}^n c_k\overline c_l \phi(k,l),
$$
so $\phi$ is positive definite, since $\omega_\phi$ is a positive functional. If we let $(e_{kl})_{k,l=0}^\infty$ denote the standard matrix units in $B(\ell^2(\N_0))$, then
$$
e_{kl} = S^k(S^*)^l - S^{k+1}(S^*)^{l+1},
$$
and so
$$
\phi(k,l) - \phi(k+1,l+1) = \omega_\phi(S^k(S^*)^l - S^{k+1}(S^*)^{l+1}) = \omega_\phi(e_{kl}).
$$
It follows that
$$
0 \leq \omega_\phi\left( \left(\sum_{k=0}^n c_k e_{k0} \right) \left(\sum_{l=0}^n c_l e_{l0} \right)^* \right) = \sum_{k,l=0}^n c_k\overline c_l(\phi(k,l) - \phi(k+1,l+1)),
$$
so $\phi - \phi\circ\sigma$ is positive definite. Finally, $\phi(0,0) = \omega_\phi(1) = 1$.

(4)$\implies$(3): Since $\phi - \phi\circ\sigma$ is positive definite, $\phi(0,0) \geq \phi(1,1) \geq \cdots$, and since $\phi$ is positive definite, $\phi(k,k) \geq 0$ for every $k$. Hence $\lim_k \phi(k,k)$ exists. Let $h = \phi - \phi\circ\sigma$. Then $h$ is positive definite, and
$$
\sum_{k=0}^\infty h(k,k) = \sum_{k=0}^\infty (\phi - \phi\circ\sigma)(k,k) = \phi(0,0) - \lim_{k\to\infty} \phi(k,k) < \infty.
$$
By the Cauchy-Schwarz inequality,
$$
\sum_{i=0}^\infty |h(k+i,l+i)| \leq \sum_{i=0}^\infty \sqrt{h(k+i,k+i)}\sqrt{h(l+i,l+i)} < \infty,
$$
so $\lim_i\phi(k+i,l+i)$ exists for every $k,l$ and depends of course only on $k-l$. Define $\phi_0(k-l) = \lim_i\phi(k+i,l+i)$. Since $\phi$ is positive definite, it follows that $\phi\circ\sigma^i$ is positive definite, so the limit $\phi_0$ is as well. Finally note that
$$
\phi(k,l) = \sum_{i=0}^\infty h(k+i,l+i) + \phi_0(k-l).
$$

(3)$\implies$(1): Let $\omega_1$ be the functional on $B(\ell^2(\N_0))$ given by $\omega_1(x) = \Tr(h^tx)$, where $h^t$ denotes the transpose of $h$. Since $h$ is positive, this is a positive, normal, linear functional. Note that $\omega_1(e_{kl}) = \Tr(e_{0l} h^t e_{k0}) = h^t(l,k) = h(k,l)$, so that
$$
\omega_1(S^k(S^*)^l) = \omega_1(e_{kl} + e_{k+1,l+1} + \cdots ) = \sum_{i=0}^\infty h(k+i,l+i).
$$

The positive definite function $\phi_0:\Z\to\R$ corresponds to a positive functional $\omega_0$ on $C(\mb T)$ given by $\omega_0(z^{k-l}) = \phi_0(k-l)$, where $z$ denotes the standard unitary generator of $C(\mb T)$. Letting $\pi:C^*(S) \to C(\mb T)$ be the quotient map as usual, we see that $\omega = \omega_1 + \omega_0\circ\pi$ is a positive linear functional on $C^*(S)$ with
$$
\omega(S^k(S^*)^l) = \sum_{i=0}^\infty h(k+i,l+i) + \phi_0(k-l) = \phi(k,l).
$$
Hence $\omega_\phi = \omega$ is the desired positive functional on $C^*(S)$. Since $\omega(1) = \phi(0,0) = 1$, it is a state.
\end{proof}

\begin{cor}\label{cor:4}
Let $\theta:\N_0\times\N_0 \to \C$ be a function. The conditions {\us(1)} and {\us(2)} below are equivalent. Moreover, {\us(1)} implies {\us(3)}.
\begin{enumerate}
	\item The function $\theta - \frac12\theta(0,0)$ is positive definite, and $\theta - \theta\circ\sigma$ is positive definite.
	\item There exist a Hilbert space $H$ with vectors $\xi_i\in H$ such that $\sum_{i=0}^\infty || \xi_i||^2 < \infty$ and a positive definite function $\theta_0:\Z\to\C$ such that $\theta$ is given as
	$$
	\theta(k,l) = \sum_{i=0}^\infty ||\xi_i||^2 + \sum_{i=0}^\infty \la \xi_{k+i},\xi_{l+i} \ra + \theta_0(0) + \theta_0(k-l), \quad k,l\in\N_0.
	$$
	\item For all $t> 0$ we have $||M_{e^{-t\theta}}|| \leq 1$.
\end{enumerate}
\end{cor}

\begin{proof}
Let $\phi = \theta - \frac12\theta(0,0)$, and observe that $\theta - \theta\circ\sigma = \phi - \phi\circ\sigma$.
The equivalence (1)$\iff$(2) follows easily from Proposition~\ref{prop:thm1} applied to $\phi$.
For the norm estimate in (3) (assuming (1)) we use
$$
|| M_{e^{-t\theta}} || = e^{-t\phi(0,0)} || M_{e^{-t\phi}} || \leq e^{-t\phi(0,0)} e^{t || M_{\phi} || } = 1,
$$
where we used Proposition~\ref{prop:thm1} to get $||M_{\phi}|| = \phi(0,0)$.

\end{proof}

Our next goal is to characterize the functions $\psi:\N_0\times\N_0\to\C$ that generate semigroups $(e^{-t\psi})_{t>0}$ so that each $e^{-t\psi}$ induces a state on the Toeplitz algebra. With Schoenberg's Theorem in mind, the result in Proposition~\ref{prop:thm2} is not surprising. But first we prove a lemma.

\begin{lem}\label{lem:3}
Let $\psi:\N_0\times\N_0\to\C$ be a function. Suppose $\psi$ is a conditionally negative definite kernel, and $\psi\circ\sigma - \psi$ is a positive definite kernel. Then there exist a Hilbert space $H$, a sequence of vectors $(\eta_i)_{i=0}^\infty$ in $H$ such that for every $m,n\in\N_0$
	$$
	\sum_{k=0}^\infty ||\eta_{m+k} - \eta_{n+k}||^2 < \infty,
	$$
	and $(\psi\circ\sigma - \psi)(m,n) = \la\eta_m,\eta_n\ra$.
\end{lem}
\begin{proof}
Let $\phi = \psi\circ\sigma - \psi$. Since by assumption $\phi$ is positive definite, there are vectors $\eta_i\in H$, where $H$ is a Hilbert space, such that $\phi(k,l) = \la\eta_k,\eta_l\ra$ for every $k,l\in\N_0$. Define
$$
\rho(k,l) = \psi(k+1,l) + \psi(k,l+1) - \psi(k,l) - \psi(k+1,l+1).
$$
Then
\begin{align*}
(\rho - \rho\circ\sigma)(k,l) &= - \phi(k+1,l) - \phi(k,l+1) + \phi(k,l) + \phi(k+1,l+1) \\
&= \la \eta_k - \eta_{k+1},\eta_l - \eta_{l+1} \ra,
\end{align*}
so $\rho - \rho\circ\sigma$ is a positive definite kernel. In particular,
\begin{align}\label{eq:rho}
\rho(0,0) \geq \rho(1,1) \geq \rho(2,2) \geq \cdots.
\end{align}
Since $\psi$ is conditionally negative definite,
$$
- \rho(k,k) = (1,-1) \begin{pmatrix}
\psi(k,k) & \psi(k,k+1) \\
\psi(k+1,k) & \psi(k+1,k+1)
\end{pmatrix}
\binom{1}{-1} \leq 0,
$$
so $\rho(k,k) \geq 0$ for every $k$. Combining this with \eqref{eq:rho} we see that $\lim_k \rho(k,k)$ exists. Hence
$$
\sum_{k=0}^\infty ||\eta_k - \eta_{k+1}||^2 = \sum_{k=0}^\infty (\rho - \rho\circ\sigma)(k,k) = \rho(0,0) - \lim_k \rho(k,k) < \infty
$$
Let $C = \left( \sum_{k=0}^\infty ||\eta_k - \eta_{k+1}||^2 \right)^{1/2}$. The triangle inequality (for the Hilbert space $H\oplus H\oplus \cdots$) yields that
$$
\left( \sum_{k=0}^\infty ||\eta_k - \eta_{k+2}||^2 \right)^{1/2} \leq C + \left( \sum_{k=0}^\infty ||\eta_{k+1} - \eta_{k+2}||^2 \right)^{1/2} \leq 2C,
$$
and similarly
$$
\left( \sum_{i=0}^\infty ||\eta_{k+i} - \eta_{l+i}||^2 \right)^{1/2} \leq |k-l| C < \infty
$$
for all $k,l\in\N_0$.
\end{proof}

\begin{prop}\label{prop:thm2}
Let $\psi:\N_0\times\N_0\to\C$ be a function. The following are equivalent.
\begin{enumerate}
	\item For all $t>0$ the function $e^{-t\psi}$ satisfies the equivalent conditions in Proposition~\ref{prop:thm1}.
	\item $\psi$ is a conditionally negative definite kernel with $\psi(0,0) = 0$, and $\psi\circ\sigma - \psi$ is a positive definite kernel.
\end{enumerate}
Moreover, if $\psi$ takes only real values, this is equivalent to the following assertion.
\begin{enumerate}
\item[{\us(3)}] There exist a Hilbert space $H$, a sequence of vectors $(\eta_i)_{i=0}^\infty$ in $H$ and a conditionally negative definite function $\psi_0:\Z\to\R$ with $\psi_0(0) = 0$ such that
	$$
	\psi(k,l) = \frac12 \left( \sum_{i=0}^{k-1} || \eta_i||^2 + \sum_{i=0}^{l-1} ||\eta_i||^2 + \sum_{i=0}^\infty || \eta_{k+i} - \eta_{l+i} ||^2 \right) + \psi_0(k-l)
	$$
	for all $(k,l)\in\N_0\times\N_0$ (and the infinite sum is convergent). 
\end{enumerate}
\end{prop}
\begin{proof}
(1)$\implies$(2): We assume that condition (4) of Proposition~\ref{prop:thm1} holds for $e^{-t\psi}$ for each $t> 0$. Then the function $e^{-t\psi}$ is a positive definite kernel for each $t> 0$, and so $\psi$ is conditionally negative definite by Schoenberg's Theorem. Since $e^{-t\psi(0,0)} = 1$ for all $t>0$, we must have $\psi(0,0) = 0$. Moreover, $e^{-t\psi} - e^{-t\psi\circ\sigma}$ is positive definite for each $t > 0$, and hence is
$$
\psi\circ\sigma - \psi = \lim_{t\to 0} \frac{e^{-t\psi} - e^{-t\psi\circ\sigma} }{t},
$$
where the limit is pointwise.

(2)$\implies$(1): For obvious reasons it suffices to prove the case $t = 1$. We verify condition (4) of Proposition~\ref{prop:thm1}. An application of Schoenberg's Theorem shows that $e^{-\psi}$ is positive definite, and of course $e^{-\psi(0,0)} = 1$. Consider
\begin{align}\label{eq:prod}
e^{-\psi} - e^{-\psi\circ\sigma} = e^{-\psi\circ\sigma}\left( e^{(\psi\circ\sigma - \psi)} - 1\right).
\end{align}
The function $e^{-\psi\circ\sigma}$ is positive definite by Schoenberg's Theorem. Expanding the exponential function in the parenthesis as a power series we get
$$
e^{\psi\circ\sigma - \psi} - 1 = \sum_{n=1}^\infty \frac{(\psi\circ\sigma - \psi)^n}{n!},
$$
and since $\psi\circ\sigma - \psi$ is positive definite, so is each power $(\psi\circ\sigma - \psi)^n$, and so is the sum, and hence also the product in \eqref{eq:prod}. The conditions in (4) of Proposition~\ref{prop:thm1} have now been verified.

Suppose $\psi$ takes only real values.

(2)$\implies$(3): By Lemma~\ref{lem:3} there are vectors $(\eta_i)_{i=0}^\infty$ in a Hilbert space $H$, such that $(\psi\circ\sigma - \psi)(m,n) = \la\eta_m,\eta_n\ra$ and
$$
\sum_{k=0}^\infty || \eta_{m+k} - \eta_{n+k} ||^2 < \infty
$$
for every $m,n\in\N_0$. Since $\psi$ is hermitian and real, it is symmetric. Hence $\la\eta_m,\eta_n\ra = \la\eta_n,\eta_m\ra$.

Let $f(k) = \frac12 \sum_{i=0}^{k-1} ||\eta_i||^2$, and set
$$
\psi_2(k,l) = \psi(k,l) - f(k) - f(l), \quad \psi_1(k,l) = \psi_2(k,l) - \frac12 \sum_{i=0}^\infty ||\eta_{k+i} - \eta_{l+i}||^2.
$$
We claim that $\psi_1$ is conditionally negative definite, $\psi_1(0,0) = 0$, and that $\psi_1(k,l)$ only depends on $k-l$. These claims will finish the proof of (2)$\implies$(3).
We find
\begin{align*}
(\psi_1\circ\sigma - \psi_1)(k,l)
&= (\psi\circ\sigma - \psi)(k,l) - \frac12 ||\eta_k||^2 - \frac12||\eta_l||^2  + \frac12 ||\eta_k - \eta_l||^2
\\ &= \la \eta_k,\eta_l\ra - \frac12 ||\eta_k||^2 - \frac12||\eta_l||^2  + \frac12 ||\eta_k - \eta_l||^2 = 0,
\end{align*}
and hence $\psi_1(k,l)$ only depends on $k-l$. Letting $\psi_0(k-l) = \psi_1(k,l)$ gives a well-defined function $\psi_0:\Z\to\R$. Note that
$$
\psi_0(0) = \psi_1(0,0) = \psi_2(0,0) = \psi(0,0) = 0.
$$
It remains to be seen that $\psi_0$ is conditionally negative definite. Observe that $\psi_2$ is conditionally negative definite, because $\psi$ is. Also,
$$
\psi_2(k,l) = \psi_0(k-l) + \frac12 \sum_{i=0}^\infty || \eta_{k+i} - \eta_{l+i} ||^2.
$$
Replacing $(k,l)$ by $(k+n,l+n)$ we see that
$$
\psi_2(k+n,l+n) = \psi_0(k-l) + \frac12 \sum_{i=n}^\infty || \eta_{k+i} - \eta_{l+i} ||^2,
$$
and so
$$
\lim_{n\to\infty} \psi_2(k+n,l+n) = \psi_0(k-l).
$$
Since $\psi_2$ was conditionally negative definite (and hence also $\psi_2\circ\sigma^n$), it follows that $\psi_0$ is conditionally negative definite being the pointwise limit of conditionally negative definite kernels.

(3)$\implies$(2): Since each of the functions
$$
(k,l)\mapsto \sum_{i=0}^{k-1} ||\eta_i||^2, \quad
(k,l)\mapsto ||\eta_{k+i} - \eta_{l+i}||^2, \quad
\psi_0
$$
is a conditionally negative definite kernel, so is $\psi$. Also $\psi(0,0) = \psi_0(0) = 0$. Finally, $(\psi\circ\sigma - \psi)(k,l) = \la\eta_k,\eta_l\ra$, and hence $\psi\circ\sigma - \psi$ is positive definite.

\end{proof}

\subsection{Decomposition into positive and negative parts}

In the following section we will describe the kernels $\phi$ such that $\|\omega_{e^{-t\phi}}\| \leq 1$ for every $t>0$. The main result here is contained in Theorem~\ref{thm:split}.

\begin{defi}\label{defi:S}
Denote by $\mc S$ the set of hermitian functions $\phi:\N_0^2\to\C $ that split as $\phi = \psi + \theta$, where
\begin{itemize}
\renewcommand{\labelitemi}{$\cdot$}
	\item $\psi$ is a conditionally negative definite kernel with $\psi(0,0) = 0$,
	\item $\psi\circ\sigma - \psi$ is a positive definite kernel,
	\item $\theta - \frac12\theta(0,0)$ is a positive definite kernel,
	\item $\theta - \theta\circ\sigma$ is a positive definite kernel.
\end{itemize}
\end{defi}
Observe that $\mc S$ is stable under addition, multiplication by positive numbers and addition by positive constant functions.

\begin{lem}\label{lem:Sclosed}
The set $\mc S$ is closed in the topology of pointwise convergence.
\end{lem}
\begin{proof}
Let $(\phi_i)_{i\in I}$ be a net in $\mc S$ converging pointwise to $\phi$, and let $\phi_i = \psi_i + \theta_i$ be a splitting guaranteed by the assumption $\phi_i\in\mc S$. 

An application of the Cauchy-Schwarz inequality to the positive definite kernel $\theta_i - \frac12\theta_i(0,0)$ gives
$$
|\theta_i(k,l) - \tfrac12\theta_i(0,0)| \leq \left(\theta_i(k,k) - \tfrac12\theta_i(0,0)\right)^{1/2}\left(\theta_i(l,l) - \tfrac12\theta_i(0,0)\right)^{1/2}
$$
and using positive definiteness of $\theta_i - \theta_i\circ\sigma$ then gives
$$
|\theta_i(k,l) - \tfrac12\theta_i(0,0)|  \leq \theta_i(0,0) - \tfrac12\theta_i(0,0) = \tfrac12\theta_i(0,0) = \tfrac12\phi_i(0,0).
$$
Since $\theta_i(0,0) = \phi_i(0,0) \to \phi(0,0)$, this shows that the net $(\theta_i(k,l))_{i\in I}$ is eventually bounded for each pair $(k,l)$. It follows that for each pair $(k,l)$ the net $\psi_i(k,l)$ is also eventually bounded.

Let $(\psi_j)_{j\in J}$ and $(\theta_j)_{j\in J}$ be universal subnets of $(\psi_i)_{i\in I}$ and $(\theta_i)_{i\in I}$ (we can assume, as we have done, that they have the same index set $J$). Since the net $(\psi_j)_{j\in J}$ is pointwise eventually bounded, it converges to some limit $\psi$. Similarly let $\theta = \lim_j \theta_j$.
Since the defining properties of the splitting $\phi_j = \psi_j + \theta_j$ pass to the limits $\psi$ and $\theta$, we have the desired splitting $\phi = \psi + \theta$, and the proof is done.
\end{proof}

We have the following alternative characterization of the set $\mc S$. This should be compared with the result in Theorem~\ref{thm:A}.

\begin{thm}\label{thm:split}
Let $\phi:\N_0^2\to\C$ be a hermitian function. Then $\phi\in\mc S$ if and only if $$||\omega_{e^{-t\phi}}|| \leq 1 \quad \text{for every }t > 0,$$ where $\omega_{e^{-t\phi}}$ is the functional associated with $e^{-t\phi}$ as in \eqref{eq:omega}.
\end{thm}

For the proof we need the following lemma.
\begin{lem}\label{lem:2}
Let $\phi:\N_0^2\to\C$ be a hermitian function. If $||\omega_\phi|| \leq 1$, then $1 - \phi\in\mc S$.
\end{lem}
\begin{proof}
Since $\phi$ is hermitian, $\omega_\phi$ is hermitian. Now, use the Hahn-Jordan Decomposition Theorem to write $\omega_\phi = \omega^+ - \omega^-$, where $\omega^+,\omega^- \in C^*(S)^*$ are positive functionals. If we define functions $\phi^+,\phi^-$ by
$$
\phi^\pm(m,n) = \omega^\pm(S^m(S^*)^n),
$$
then $\phi^+$ and $\phi^-$ satisfy the second condition of Proposition~\ref{prop:thm1} (up to a scaling factor). Also, it is clear that $\phi = \phi^+ - \phi^-$.

Let $c = \phi^+(0,0)$, and put
$$
\psi = c - \phi^+, \quad \theta = 1-c + \phi^-.
$$
Obviously, $\psi + \theta = 1 - \phi$. It remains to show that $\psi$ and $\theta$ have the desired properties used in the definition of $\mc S$.

Since $\phi^+$ is positive definite, $\psi$ is conditionally negative definite with $\psi(0,0) = 0$. Also, $\psi\circ\sigma - \psi = \phi^+ - \phi^+\circ\sigma$, which is positive definite by Proposition~\ref{prop:thm1}.

Moreover, $\theta - \theta\circ\sigma = \phi^- - \phi^-\circ\sigma$ is positive definite. Finally,
$$
\theta - \frac12 \theta(0,0) = \frac12 ( 1-\phi(0,0)) + \phi^-,
$$
and since $|\phi(0,0)| \leq ||\omega_\phi|| \leq 1$ and $\phi(0,0)$ is real, it follows that $1 - \phi(0,0) \geq 0$, so $\theta - \frac12 \theta(0,0)$ is positive definite (using that $\phi^-$ is positive definite).
\end{proof}

\begin{proof}[Proof of Theorem~\ref{thm:split}]
Suppose first $||\omega_{e^{-t\phi}}|| \leq 1$ for all $t> 0$. It follows from the previous lemma that $1- e^{-t\phi} \in\mc S$ for every $t> 0$. Hence the functions $(1 - e^{-t\phi})/t$ are in $\mc S$, and they converge pointwise to $\phi$ as $t\to 0$. Since $\mc S$ is closed under pointwise convergence, we conclude that $\phi\in\mc S$.

Conversely, suppose $\phi\in\mc S$. Write $\phi = \psi + \theta$ as in the definition of $\mc S$. From Proposition~\ref{prop:thm2} we get that $M_{e^{-t\psi}}$ is a u.c.p. map for every $t>0$, and hence $||M_{e^{-t\psi}}|| = 1$. Also, from Corollary~\ref{cor:4} we get that $||M_{e^{-t\theta}}|| \leq 1$ for every $t>0$. This combines to show
$$
||\omega_{e^{-t\phi}}|| = ||M_{e^{-t\phi}}|| \leq ||M_{e^{-t\psi}}|| \ ||M_{e^{-t\theta}}|| \leq 1.
$$

\end{proof}

\subsection{Comparison of norms}

In this section we establish the connection between norms of radial Herz-Schur multipliers on $\F_\infty$ and functionals on the Toeplitz algebra. This will be accomplished in Proposition~\ref{prop:norms}.

In \cite{HSS} the following theorem is proved (see Theorem 5.2 therein).
\begin{thm}\label{thm:HSS}
Let $\F_\infty$ be the free group on (countably) infinitely many generators, let $\phi:\F_\infty\to \C$ be a radial function, and let $\dot\phi:\N_0\to\C$ be as in Definition~\ref{defi:radial}. Finally, let $h=(h_{ij})_{i,j\in\N_0}$ be the Hankel matrix given by $h_{ij}=\dot\varphi(i+j)-\dot\varphi(i+j+2)$ for $i,j\in\N_0$. Then the following are equivalent:
	\begin{itemize}
		\item [(i)]$\phi$ is a Herz-Schur multiplier on $\F_\infty$.
		\item [(ii)]$h$ is of trace class.
	\end{itemize}
	If these two equivalent conditions are satisfied, then there exist unique constants $c_\pm\in\C$ and a unique $\dot\psi:\N_0\to\C$ such that
	\begin{align}\label{eq:pres}
		\dot\phi(n)=c_++c_-(-1)^n+\dot\psi(n)\qquad(n\in\N_0)
	\end{align}
	and
	\begin{equation*}
		\lim_{n\to\infty}\dot\psi(n)=0.
	\end{equation*}
	Moreover,
	\begin{equation*}
		\|\varphi\|_\HS=|c_+|+|c_-|+\|h\|_1.
	\end{equation*}
\end{thm}

The Fourier-Stieltjes algebra $B(\Z)$ of the group of integers is the linear span of positive definite functions on $\Z$. It is naturally identified with dual space of $C^*(\Z) \simeq C(\mb T)$, i.e., with the set $M(\mb T)$ of complex Radon measures on the circle, where $\phi\in B(\Z)$ corresponds to $\mu\in M(\mb T)$, if and only if
\begin{align}\label{eq:BM}
\phi(n) = \int_{\mb T} z^n d\mu(z) \quad\text{for all } n\in\Z.
\end{align}
Under this identification $B(\Z)$ becomes a Banach space when the norm $||\phi||_{B(\Z)}$ is defined to be $||\mu||$, the total variation of $\mu$.

\begin{prop}\label{prop:HSS28}
Let $\phi:\N_0\times\N_0\to\C$ be a function, and let $h = \phi - \phi\circ\sigma$. The functional $\omega_\phi$ extends to a bounded functional on $C^*(S)$ if and only if $h$ is of trace class, and the function $\phi_0:\Z\to\C$ given by $\phi_0(m-n) = \lim_k \phi(m+k,n+k)$ (which is then well-defined) lies in $B(\Z)$. If this is the case, then
$$
||\omega_\phi|| = ||h||_1 + ||\phi_0||_{B(\Z)}.
$$
\end{prop}
\begin{proof}
The proposition is actually a special case of a general phenomenon. Given an extension ${0 \rightarrow I \rightarrow A \rightarrow A/I \rightarrow 0}$ of \Cs-algebras, then $A^* \simeq I^* \oplus_1 (A/I)^*$ isometrically. The extension under consideration in our proposition is \eqref{eq:extension}. The general theory is described in the book \cite{MR1238713}. We have included a more direct proof.

Suppose first $h$ is of trace class, and $\phi_0\in B(\Z)$. Let $\mu\in M(\mb T)$ be given by \eqref{eq:BM}, and define $\omega_0 \in C(\mb T)^*$ by
$$
\omega_0(f) = \int_{\mb T} f d\mu \quad\text{for all } f\in C(\mb T).
$$
Define a functional $\omega_1$ on $C^*(S)$ by $\omega_1(x) = \Tr(h^t x)$ for $x\in C^*(S)$, and also let $\omega = \omega_1 + \omega_0\circ\pi$. Observe that
$$
\omega_1(S^m(S^*)^n) = \sum_{k=0}^\infty h(m+k,n+k).
$$
It follows that
$$
\omega(S^m(S^*)^n) = \sum_{k=0}^\infty h(m+k,n+k) + \phi_0(m-n) = \phi(m,n),
$$
so that $\omega = \omega_\phi$. Hence $\omega_\phi$ extends to a bounded functional on $C^*(S)$.

Suppose instead that $\omega_\phi$ extends to a bounded functional on $C^*(S)$. Proposition 2.8 in \cite{HSS} ensures the existence of a complex Borel measure $\mu$ on $M(\mb T)$ and a trace class operator $T$ on $\ell^2(\N_0)$ such that
$$
\omega_\phi(S^m(S^*)^n) = \int_{\mb T} z^{m-n} d\mu(z) + \Tr(S^m(S^*)^nT) \quad\text{for all } m,n\in\N_0.
$$
From this we get that $T^t_{mn} = h(m,n)$, where $T^t$ is the transpose of $T$, and
$$
\phi_0(m-n) = \int_{\mb T} z^{m-n} d\mu(z).
$$
Hence $h$ is of trace class and $\phi_0\in B(\Z)$. From \cite{HSS} we also have $||\omega_\phi|| = ||\mu|| + ||T||_1$, which concludes our proof, since
$$
||\mu|| = ||\phi_0|| \quad\text{and}\quad ||T^t||_1 = ||h||_1.
$$
\end{proof}

\begin{prop}\label{prop:norms}
Let $\phi:\F_\infty\to\C$ be a radial function, and let $\tilde\phi:\N_0 \times\N_0\to\C$ be as in Definition~\ref{defi:radial}. Then $\|\omega_{\tilde\phi}\| = \|\phi\|_\HS$.
\end{prop}
\begin{proof}
Let $h_{ij} = \tilde\phi(i,j) - \tilde\phi(i+1,j+1)$. From Theorem~\ref{thm:HSS} and Proposition~\ref{prop:HSS28} we see that it suffices to consider the case where $h$ is the matrix of a trace class operator, since otherwise $\|\omega_{\tilde\phi}\| = \|\phi\|_\HS = \infty $. If $h$ is of trace class, then we let $\tilde\phi_0(n) = \lim_k \tilde\phi(k+n,k)$. From Theorem~\ref{thm:HSS} and Proposition~\ref{prop:HSS28} it follows that
\begin{align*}
\|\omega_{\tilde\phi}\| &= \| h \|_1 + \|\tilde\phi_0\|_{B(\Z)} \\
\|\phi\|_\HS &= \| h \|_1 + |c_+| + |c_-|,
\end{align*}
where $c_\pm$ are the constants obtained in Theorem~\ref{thm:HSS}. It follows from \eqref{eq:pres} that
$$
\tilde\phi_0(n) = c_+ + (-1)^n c_-.
$$
Now we only need to see why $|c_+| + |c_-| = \|\tilde\phi_0\|_{B(\Z)}$. Let $\nu\in C(\mb T)^*$ be the functional given by $\nu(f) = c_+f(1) + c_-f(-1)$ for all $f\in C(\mb T)$. Observe that $\nu(z\mapsto z^n) = c_+ + (-1)^nc_-$. Hence $\nu$ corresponds to $\tilde\phi_0$ under the isometric isomorphism $B(\Z) \simeq C(\mb T)^*$. It is easily seen that $\|\nu\| = |c_+| + |c_-|$. So
$$
\|\tilde\phi_0\|_{B(\Z)} = |c_+| + |c_-|.
$$
This completes the proof.
\end{proof}

\subsection{The linear bound}

We now restrict our attention to functions $\phi:\N_0^2\to\C$ of the form $\phi(m,n) = \dot\phi(m+n)$ for some function $\dot\phi:\N_0\to\C$. Recall that such functions are called \emph{Hankel functions}.

A function $\phi:\N_0\times\N_0\to\C$ is called \emph{linearly bounded} if there are constants $a,b\geq 0$ such that $|\phi(m,n)| \leq b + a(m+n)$ for all $m,n\in\N_0$.
\begin{prop}\label{prop:bound}
If $\phi$ is a Hankel function and $\phi\in\mc S$, then $\phi$ is linearly bounded.
\end{prop}
\begin{proof}
Write $\phi = \psi + \theta$ as in Definition~\ref{defi:S}. Note that
$$
\phi\circ\sigma - \phi = (\psi\circ\sigma - \psi) - (\theta - \theta\circ\sigma),
$$
where $h_1 = \psi\circ\sigma - \psi$ and $h_2 = \theta - \theta\circ\sigma$ are positive definite. As in the definition of a Hankel function, write $\phi(m,n) = \dot\phi(m+n)$, and let $\dot h(m) = \dot\phi(m+2) - \dot\phi(m)$, so that
$$
\dot h(m+n) = (\phi\circ\sigma - \phi)(m,n) = h_1(m,n) - h_2(m,n).
$$
We will now prove that $\dot h$ is bounded, and this will lead to the conclusion of the proposition.

From Corollary~\ref{cor:4} and Lemma~\ref{lem:3} we see that there are vectors $\xi_k,\eta_k$ in a Hilbert space such that
$$
h_1(m,n) = \la \eta_m,\eta_n\ra,
\quad
h_2(m,n) = \la \xi_m,\xi_n\ra
$$
for all $m,n\in\N_0$, and
$$
\sum_{k=0}^\infty ||\eta_k - \eta_{k+1}||^2 < \infty,
\quad
\sum_{k=0}^\infty ||\xi_k||^2 < \infty.
$$
From this we see that $h_2$ is the matrix of a positive trace class operator. Also, there is $c > 0$ such that $||\eta_k - \eta_{k+1}|| \leq c$ for every $k$ and $||\eta_0|| \leq c$, so we get the linear bound $||\eta_k|| \leq c(k+1)$. From the Cauchy-Schwarz inequality we get
\begin{align}\label{eq:h1}
|h_1(m,n)| \leq c^2(m+1)(n+1).
\end{align}

Since $h_2 \leq ||h_2|| I \leq ||h_2||_1 I$ (as positive definite matrices, where $I$ is the identity operator), we deduce that the function
$$
(m,n)\mapsto \dot h(m+n) + ||h_2||_1\delta_{mn} = h_1(m,n) + (||h_2||_1\delta_{mn} - h_2(m,n))
$$
is a positive definite kernel. By the Cauchy-Schwarz inequality we have
\begin{align}\label{eq:cs}
|\dot h(k)|^2 \leq (\dot h(0) + d)(\dot h(2k) + d)
\end{align}
for every $k\geq 1$, where $d = ||h_2||_1$. If $e = \dot h(0) + d$ is zero, then clearly $\dot h(k) = 0$ when $k\geq 1$. Suppose $e > 0$. Then we may rewrite \eqref{eq:cs} as
$$
\dot h(2k) \geq \frac{|\dot h(k)|^2}{e} - d.
$$
We claim that $\dot h$ is bounded by $2e+d$. Suppose by contradiction that $|\dot h(k_0)| > 2e + d$ for some $k_0 \geq 1$. Then by induction over $n$ we may prove that for any $n\in\N_0$
\begin{align}\label{eq:exp}
\dot h(k_0 2^{n+1}) \geq 2^{2^n} (2e+d).
\end{align}
For $n = 0$ we have
$$
\dot h(2k_0) \geq \frac{|\dot h(k_0)|^2}{e} - d \geq \frac{(2e+d)^2}{e} - d \geq 4e + 3d \geq 2(2e+d).
$$
For $n\geq 1$ we get (using our induction hypothesis)
\begin{align*}
\dot h(k_02^{n+1}) & \geq \frac{|\dot h(k_02^n)|^2}{e} - d \geq (2^{2^{n-1}} )^2\frac{(2e+d)^2}{e} - d \\
& \geq 2^{2^n}(4e + 4d) - d \geq 2^{2^n}(2e+d).
\end{align*}
Using \eqref{eq:h1} we observe that for every $m\in\N_0$ we have
$$
|\dot h(m) | \leq |h_1(m,0)| + |h_2(m,0)| \leq c^2(m+1) + d.
$$
Since $e>0$, this contradicts \eqref{eq:exp}. This proves the claim. It follows that
$$
|\dot\phi(2k)| \leq |\dot\phi(0)| + (2e + d)k
$$
and
$$
|\dot\phi(2k+1)| \leq |\dot\phi(1)| + (2e + d)k.
$$
With $b = \max\{|\dot\phi(0)|,|\dot\phi(1)|\}$ and $a = 2e+d$ this shows that
$$
|\dot\phi(m)| \leq b + am,
$$
and thus $|\phi(m,n)| \leq b + a(m+n)$, which proves the proposition.

\end{proof}

\begin{thm}\label{thm:main-finite}
If $\phi:\F_\infty\to\R$ is a radial function such that $\|e^{-t\phi}\|_\HS \leq 1$ for each $t> 0$, then there are constants $a,b\geq0$ such that $\phi(x) \leq b + a|x|$ for all $x\in\F_\infty$. Here $|x|$ denotes the word length function on $\F_\infty$.
\end{thm}
\begin{proof}
Suppose $\phi:\F_\infty\to\R$ is a radial function such that $\|e^{-t\phi}\|_\HS \leq 1$ for each $t> 0$, and let $\tilde\phi$ be as in Definition~\ref{defi:radial}. First observe that $\tilde\phi$ is real, symmetric, and thus hermitian. From Proposition~\ref{prop:norms} we get that $\| \omega_{e^{-t\tilde\phi}} \|_{B_2} \leq 1$ for every $t > 0$, so Theorem~\ref{thm:split} implies that $\tilde\phi\in \mc S$. Since $\tilde\phi$ is a Hankel function, Proposition~\ref{prop:bound} ensures that
$$
|\tilde\phi(m,n)| \leq b + a(m+n) \quad\text{for all } m,n\in\N_0
$$
for some constants $a$ and $b$. This shows that
$$
\phi(x) \leq b + a|x| \quad\text{for all } x\in\F_\infty.
$$
\end{proof}

This finishes the proof of Theorem~\ref{thm:B} in the case of the free group on infinitely many generators.

\section{Radial semigroups of Herz-Schur multipliers on \texorpdfstring{$\F_n$}{Fn}}\label{sec:finite}

The proof of Theorem~\ref{thm:B} for the finitely generated free groups is more technical than the proof concerning $\F_\infty$, but the general approach is the same, and most of the steps in the proof can be deduced from what we have already done for $\F_\infty$. In order to do so we introduce the transformations $F$ and $G$ that, loosely speaking, translate between the two cases, the finite and the infinite.

\subsection{The transformations \texorpdfstring{$F$}{F} and \texorpdfstring{$G$}{G}}\label{sec:FG}

From now on we fix a natural number $q$ with $2\leq q < \infty$. If the number of generators of the free group under consideration is $n$, we will let $q = 2n - 1$. The parametrization using $q$ instead of $n$ is adapted from \cite{HSS}. As before, the free groups will enter the picture quite late (Theorem~\ref{thm:q-HSS}), and we will mainly focus on functions $\phi:\N_0\times\N_0\to\C$.

We still denote the unilateral shift operator on $\ell^2(\N_0)$ by $S$. For each $m,n\in\N_0$ we let $S_{m,n}$ denote the operator
$$
S_{m,n} = \left(1-\frac1q\right)^{-1}\left(S^m(S^*)^n - \frac1q S^*S^m(S^*)^nS\right).
$$
Observe that
\begin{align}\label{eq:Smn}
S_{m,n} = \begin{cases}
S^m(S^*)^n & \text{if } \min\{m,n\}=0, \\
\left(1-\frac1q\right)^{-1}\left(S^m(S^*)^n - \frac1q S^{m-1}(S^*)^{n-1}\right) & \text{if } m,n\geq 1,
\end{cases}
\end{align}
and
\begin{align}\label{eq:recursion}
S^m(S^*)^n = \left( 1 - \frac1q \right)S_{m,n} + \frac1q S^{m-1}(S^*)^{n-1}, \quad\text{when } m,n\geq 1.
\end{align}
It follows by induction over $\min\{m,n\}$ that $S^m(S^*)^n \in \mr{span}\{S_{k,l} \mid k,l\in\N_0\}$ for all $m,n\in\N_0$, so $\mr{span}\{S_{k,l} \mid k,l\in\N_0\} = D$, where $D$ is given by \eqref{eq:D}.

When $\phi:\N_0^2\to\C$ is a function we let $\chi_\phi$ denote the linear functional defined on $D$ by
$$
\chi_\phi(S_{m,n}) = \phi(m,n),
$$
and if it extends to a bounded functional on $C^*(S)$, we also denote the extension by $\chi_\phi$.

Let $\mc V$ be the set of kernels on $\N_0$, that is, $\mc V = \C^{\N_0\times\N_0} = \{ \phi:\N_0^2\to\C \}$. Then $\mc V$ is a vector space over $\C$ under pointwise addition and scalar multiplication. We equip $\mc V$ with the topology of pointwise convergence.

Recall that $\sigma:\N_0^2 \to \N_0^2$ is the shift map $\sigma(k,l) = (k+1,l+1)$. We now define operators $\tau$ and $\tau^*$ on $\mc V$. For $\phi\in\mc V$ the operator $\tau^*$ is given by $\tau^*(\phi) = \phi\circ\sigma$, and $\tau(\phi)$ given by
$$
\tau(\phi)(k,l) = \begin{cases}
\phi(k-1,l-1), & k,l\geq 1, \\
0, & \min\{k,l\} = 0.
\end{cases}
$$
Then $\tau^*\circ\tau = \id$, and $(\tau\circ\tau^*)(\phi) = 1_{\N\times\N} \phi$. We have the following rules
$$
(\tau\circ\tau^*)^2 = \tau\circ\tau^*, \quad \tau^*\circ\tau\circ\tau^* = \tau^*, \quad \tau\circ\tau^*\circ\tau = \tau.
$$

Each element of $B(\ell^2(\N_0))$ may be identified with its matrix representation (with respect to the canonical orthonormal basis) and may in this way be considered as an element of $\mc V$. Under this identification $\tau$ and $\tau^*$ restrict to maps on $B(\ell^2(\N_0))$ given by $\tau(A) = SAS^*$ and $\tau^*(A) = S^*AS$. Clearly, $\tau$ is an isometry on the bounded operators. As noted in \citep{HSS}, it is also an isometry on the trace class operators $B_1(\ell^2(\N_0))$ (with respect to the trace norm). The operator
$$
\left( 1 - \frac\tau\alpha \right)^{-1} = \sum_{n=0}^\infty \frac{\tau^n}{\alpha^n}
$$
on $B_1(\ell^2(\N_0))$ is therefore well-defined when $\alpha > 1$, and its norm is bounded by $(1 - \frac1\alpha)^{-1}$. To shorten notation we let
\begin{align}\label{eq:F}
F = \left(1 - \frac1q\right) \left( \id - \frac\tau q \right)^{-1} = \left(1 - \frac1q\right)\sum_{n=0}^\infty \frac{\tau^n}{q^n}.
\end{align}

We note that $F$ defined by \eqref{eq:F} also makes sense as an invertible operator on $\mc V$ as $(\tau^n\phi)(k,l) = 0$ for $n \gg 0$.

Let $G$ be the operator on $\mc V$ defined recursively by
$$
G\phi(m,n) = \phi(m,n) \quad\text{if } \min\{m,n\} = 0
$$
and
$$
G\phi(m,n) = \left(1-\frac1q\right) \phi(m,n) + \frac1q G\phi(m-1,n-1), \quad\text{if } m,n\geq 1,
$$
when $\phi\in\mc V$. We can reconstruct $\phi$ from $G\phi$. In fact $G^{-1}\phi$ is given by
\begin{align}\label{eq:Ginv}
G^{-1}\phi(k,l) = \begin{cases}
\phi(k,l), & \min\{k,l\} = 0, \\
\left(1 - \frac1q\right)^{-1} \left(\phi(k,l) - \frac1q\ \phi(k-1,l-1)\right), & k,l\geq 1
\end{cases}
\end{align}

We may express $G^{-1}$ in terms of $\tau$ and $\tau^*$. We have
\begin{align}\label{eq:G}
G^{-1} = \id + \frac{1}{q-1} (\tau\circ\tau^* - \tau) = \left(\id + \frac{1}{q-1}\  \tau\circ\tau^*\right) \circ \left(\id - \frac1q \tau\right).
\end{align}
Using $(\tau\circ\tau^*)^n = \tau\circ\tau^*$ we obtain
$$
\left(\id - \frac1q \tau\circ\tau^*\right)^{-1} = \id + \left( \frac1q + \frac{1}{q^2} + \cdots \right)\tau\circ\tau^* = \id + \frac{1}{q-1}\  \tau\circ\tau^*,
$$
so it follows that
$$
G =   \left(\id - \frac1q \tau\right)^{-1} \circ \left(\id + \frac{1}{q-1}\  \tau\circ\tau^*\right)^{-1} =  \left(\id - \frac1q \tau\right)^{-1} \circ \left(\id - \frac1q\ \tau\circ\tau^*\right).
$$

We now record some elementary facts about $F$ and $G$, which will be used later on without reference.
\begin{lem}
The transformations $F$ and $G$ are linear and continuous on $\mc V$. Furthermore, $G$ takes the constant function $1$ to itself. Also, $G\phi(0,0) = \phi(0,0)$ for every $\phi\in\mc V$.
\end{lem}

\begin{lem}\label{lem:Glim}
Let $\phi\in\mc V$, and fix $m,n\in\N$. If $\lim_{k\to\infty} G\phi(m+k,n+k)$ exists, then $\lim_{k\to\infty} \phi(m+k,n+k)$ exists, and the limits are equal.
\end{lem}
\begin{proof}
We may assume that $m,n\geq 1$. Suppose $\lim_k G\phi(m+k,n+k)$ exists. Since
$$
\phi(m+k,n+k) = \left(1 - \frac1q\right)^{-1} \left(G\phi(m+k,n+k) - \frac1q\ G\phi(m+k-1,m+k-1)\right)
$$
when $k\geq 1$, we see that the limit $\lim_k \phi(m+k,n+k)$ exists and is equal to $\lim_k G\phi(m+k,n+k)$.
\end{proof}

\begin{lem}\label{lem:sa}
The transformations $F$ and $G$ both take hermitian functions to hermitian functions. So do their inverses $F^{-1}$ and $G^{-1}$.
\end{lem}
\begin{proof}
For $G$ this easily follows from inspecting the definition. For $G^{-1}$ simply look at \eqref{eq:Ginv}. For $F$ it suffices to note that $\tau^n$ preserves hermitian functions for each $n$, and hence does the sum in \eqref{eq:F}. For $F^{-1}$ it suffices to note that $\tau$ preserves hermitian functions.
\end{proof}

\begin{lem}\label{lem:FG}
There is the following relationship between $F$ and $G$.
$$
(1 - \tau^*)\circ G = F\circ(1-\tau^*).
$$

In other words, $G\phi - G\phi\circ\sigma = F(\phi - \phi\circ\sigma)$ for every $\phi\in\mc V$.
\end{lem}
\begin{proof}
It is equivalent to show $F^{-1} \circ(1 - \tau^*) = (1-\tau^*)\circ G^{-1}$, and this is easy using \eqref{eq:F} and \eqref{eq:G}.
\end{proof}

\begin{prop}\label{prop:translate}
For any $\phi:\N_0^2\to\C$ we have $\chi_\phi = \omega_{G\phi}$, where $\omega_{G\phi}$ is as in \eqref{eq:omega}.
\end{prop}
\begin{proof}
It is enough to show that $\chi_\phi$ attains the value $G\phi(m,n)$ at $S^m(S^*)^n$ for every $m,n\in\N_0$. Observe that if $\min\{m,n\} = 0$ we obviously have
$$
\chi_\phi(S^m(S^*)^n) = \chi_\phi(S_{m,n}) = \phi(m,n) = G\phi(m,n).
$$
Inductively, for $m,n\geq 1$ we get using \eqref{eq:recursion} that
\begin{align*}
\chi_\phi(S^m(S^*)^n)
   &= \left(1 - \frac1q\right) \chi_\phi(S_{m,n}) + \frac1q \chi_\phi(S^{m-1}(S^*)^{n-1})
\\ &= \left(1 - \frac1q\right) \phi(m,n) + \frac1q G\phi(m-1,n-1)
\\ &= G\phi(m,n).
\end{align*}
\end{proof}

The following proposition is analogous to Proposition~\ref{prop:HSS28}, and the proof is to deduce it from Proposition~\ref{prop:HSS28} by using transformations $F$ and $G$.
\begin{prop}\label{prop:q-HSS28}
Let $\phi:\N_0\times\N_0\to\C$ be a function, and let $h = \phi - \phi\circ\sigma$. The functional $\chi_\phi$ extends to a bounded functional on $C^*(S)$ if and only if $h$ is of trace class, and the function $\phi_0:\Z\to\C$ given by $\phi_0(m-n) = \lim_k \phi(m+k,n+k)$ (which is then well-defined) lies in $B(\Z)$. If this is the case, then
$$
||\chi_\phi|| = ||Fh||_1 + ||\phi_0||_{B(\Z)},
$$
where $F$ is the operator defined in \eqref{eq:F}.
\end{prop}
\begin{proof}
From Proposition~\ref{prop:translate} we know that $\chi_\phi = \omega_{G\phi}$. From the characterization in Proposition~\ref{prop:HSS28} we deduce that $\chi_\phi$ extends if and only if $G\phi - G\phi\circ\sigma$ is of trace class and the function $\phi'_0$ given by
$$
\phi'_0(m-n)= \lim_{k\to\infty} G\phi(m+k,n+k) \quad (m,n\in\N_0)
$$
lies in $B(\Z)$. Recall (Lemma~\ref{lem:FG}) that $G\phi - G\phi\circ\sigma = F(\phi - \phi\circ\sigma) = Fh$, and $h$ is of trace class if and only if $F(h)$ is of trace class. Also from Lemma~\ref{lem:Glim} we see that $\phi'_0 = \phi_0$.

It remains to show the norm equality. We have from Proposition~\ref{prop:HSS28}
$$
\|\chi_\phi\| = \|\omega_{G\phi}\| = \| Fh \|_1 + \|\phi'_0\|_{B(\Z)},
$$
and this completes to proof.
\end{proof}

\subsection{Relation between multipliers and functionals}
Similarly to how we defined $\chi_\phi$ as the analogue of $\omega_\phi$ we will now define $N_\phi$ as the analogue of $M_\phi$. More precisely, let $N_\phi$ be the linear map defined on $D$ by
$$
N_\phi(S_{m,n}) = \phi(m,n)S_{m,n}.
$$
It may or may not happen that $N_\phi$ extends to $C^*(S)$, and if it does we will also denote the extension by $N_\phi$. It turns out that this happens exactly when $\chi_\phi$ extends (see Proposition~\ref{prop:eq-norm}), but the proof is not as easy as the case with $M_\phi$ and $\omega_\phi$. The reason is that there is no $^*$-homomorphism $\alpha : C^*(S) \to C^*(S\tensor S)$ that maps $S_{m,n}$ to $S_{m,n} \tensor S_{m,n}$. So we cannot directly follow the approach of Remark~\ref{rem:ucp}.

Observe that $\chi_\phi =\ev_1\circ \pi \circ N_\phi$, where $\ev_1$ and $\pi$ are as in Remark~\ref{rem:ucp}. Hence $||\chi_\phi || \leq || N_\phi||$. We will now prove the reverse inequality (Proposition~\ref{prop:eq-norm}). The proof is partly contained in the proof of Theorem 2.3 in \cite{HSS}, so we will refer to that proof and emphasize the differences.

The overall strategy of our proof is the following. We find an isometry $U$ on a Hilbert space $\ell^2(X)$ and a function $\tilde\phi:X\times X\to\C$ such that $\tilde\phi$ is a Schur multiplier. We construct them in such a way that we may find a $^*$-isomorphism $\beta$ between $C^*(S)$ and $C^*(U)$ such that $N_\phi = \beta^{-1} \circ m_{\tilde\phi} \circ \beta$, where $m_{\tilde\phi}$ is the multiplier corresponding to $\tilde\phi$. The construction is similar to the one in \cite{HSS}.

Let $X$ be a homogeneous tree of degree $q+1$, i.e., each vertex has degree $q+1$. We will identify the vertex set with $X$. We fix an infinite, non-returning path $\omega = (x_0,x_1,x_2,\ldots)$ in $X$, i.e., $x_i \neq x_j$ when $i\neq j$. Define the map $c:X\to X$ such that for any $x\in X$ the sequence $x,c(x),c^2(x),\ldots$ is the unique infinite, non-returning path eventually following $\omega$. This path is denoted $[x,\omega[$. Visually, $c$ is the ``contraction'' of the tree towards the boundary point $\omega$.

\begin{defi}\label{defi:mn}
Observe that for each pair of vertices $x,y \in X$ there are smallest numbers $m,n\in\N_0$ such that $c^m(x) \in [y,\omega[$ and $c^n(y)\in [x,\omega[$. When we need to keep track of more than two points at a time, we denote $m$ and $n$ by $m(x,y)$ and $n(x,y)$ respectively.

Given a function $\phi:\N_0^2\to\C$ we define the function $\tilde\phi:X\times X\to \C$ by
$$
\tilde\phi(x,y) = \phi(m,n).
$$
\end{defi}

\begin{lem}\label{lem:additive}
Using the notation of Definition~\ref{defi:mn} we have
$$
m(x,y) - n(x,y) = m(x,z) - n(x,z) + m(z,y) - n(z,y)
$$
for every $x,y,z\in X$.
\end{lem}
\begin{proof}
Let $v\in X$ be a point sufficiently far out in $\omega$ such that $v$ lies beyond the following three points:
$$
c^{m(x,y)}(x) = c^{n(x,y)}(y),\ c^{m(x,z)}(x) = c^{n(x,z)}(z),\ c^{m(z,y)}(z) = c^{n(z,y)}(y).
$$
If we let $d$ denote the graph distance, then
\begin{align}
\begin{aligned}\label{eq:d1}
m(x,y) - n(x,y) &= m(x,y) + d(c^{m(x,y)}(x),v) - d(c^{n(x,y)}(y),v) - n(x,y) \\ &= d(x,v) - d(y,v),
\end{aligned}
\end{align}
and similarly
\begin{align}\label{eq:d2}
m(x,z) - n(x,z) = d(x,v) - d(z,v), \quad m(z,y) - n(z,y) = d(z,v) - d(y,v).
\end{align}
The lemma now follows by combining \eqref{eq:d1} and \eqref{eq:d2}.
\end{proof}

\begin{lem}\label{lem:phi0}
Let $\phi_0\in B(\Z)$ be given, and let $\phi(m,n) = \phi_0(m-n)$. The function $\tilde\phi$ from Definition~\ref{defi:mn} is a Schur multiplier, and
$$
\|\tilde\phi\|_S \leq \|\phi_0\|_{B(\Z)}.
$$
\end{lem}
\begin{proof}
It is enough to prove the lemma when $\|\phi_0\|_{B(\Z)} \leq 1$. Write $\phi_0$ in the form
$$
\phi_0(n) = \int_{\mb T} z^n d\mu(z) \quad n\in\Z,
$$
for some complex Radon measure $\mu$ on $\mb T$. First assume that $\mu = \delta_s$ for some $s\in\mb T$, so that $\phi_0$ is of the form $\phi_0(n) = s^n$. Then $\phi_0$ is a group homomorphism, so from Lemma~\ref{lem:additive} we get
$$
\tilde\phi(x,y) = \tilde\phi(x,z)\tilde\phi(z,y) \quad x,y,z\in X.
$$
In particular, if we fix some vertex, say $x_0$ from before, we have
$$
\tilde\phi(x,y) = \tilde\phi(x,x_0)\bar{\tilde\phi(y,x_0)},
$$
so $\tilde\phi$ is a positive definite kernel on $X$. Since also $\tilde\phi(x,x) = 1$ for every $x\in X$, it follows that $\tilde\phi$ is a Schur multiplier with norm at most 1 (see \cite[Theorem D.3]{BO}).

It follows that if $\mu$ lies in the set $C = \conv\{ c\delta_s \mid c,s\in\mb T\}$, then $\tilde\phi$ is a Schur multiplier of norm at most 1.

Now, let $\mu$ in $M(\mb T)_1$ be arbitrary. It follows from the Hahn-Banach Theorem that the vague closure of the set $C$ is $M(\mb T)_1$, so there is a net $(\mu_\alpha)_{\alpha\in A}$ in $C$ such that $\mu_\alpha\to\mu$ vaguely, that is,
$$
\int_{\mb T} f d\mu_\alpha \to \int_{\mb T} f d\mu \quad \text{as }\alpha\to\infty
$$
for each $f\in C(\mb T)$. In particular,
$$
\int_{\mb T} z^n d\mu_\alpha(z) \to \phi_0(n) \quad \text{as }\alpha\to\infty,
$$
so $\tilde\phi$ is the pointwise limit of Schur multipliers with norm at most 1. The proof is now complete, since the Schur multipliers of norm at most 1 are closed under pointwise convergence.

\end{proof}

Let $U$ be the operator on $\ell^2(X)$ defined by
$$
U\delta_x = \frac{1}{\sqrt q} \sum_{c(z)=x} \delta_z,
$$
where $(\delta_x)_{x\in X}$ are the standard basis vectors in $\ell^2(X)$, and let $U_{m,n}$ be defined similarly to how we defined $S_{m,n}$:
$$
U_{m,n} = \left( 1 - \frac1q \right)^{-1} \left( U^m(U^*)^n - \frac1q U^*U^m(U^*)^nU \right).
$$
It is shown in \cite{HSS} that $U$ is a proper isometry. Also if $\tilde\phi$ is a Schur multiplier, then $C^*(U)$ is invariant under $m_{\tilde\phi}$, and
$$
m_{\tilde\phi}(U_{m,n}) = \phi(m,n)U_{m,n}.
$$
Actually the authors only state the mentioned result under the assumption that $\phi(m,n)$ depends only on $m+n$, but the proof without this assumption is exactly the same (see Lemma 2.6 and Corollary 2.7 in \cite{HSS}).

Since $U$ is a proper isometry, there is a $^*$-isomorphism $\beta: C^*(S) \to C^*(U)$ such that $\beta(S) = U$. It follows that $\beta(S_{m,n}) = U_{m,n}$ for all $m,n\in\N_0$.

\begin{prop}\label{prop:eq-norm}
Let $\phi:\N_0^2\to\C$ be a function. Then $\chi_\phi$ extends to a bounded functional on $C^*(S)$ if and only if $N_\phi$ extends to a (completely) bounded map on $C^*(S)$, and in this case
$$||N_\phi|| = ||\chi_\phi||.$$
\end{prop}
\begin{proof}
As mentioned earlier it suffices to prove that if $\chi_\phi$ extends to a bounded functional on $C^*(S)$, then $N_\phi$ extends to a bounded map on $C^*(S)$ as well, and $||N_\phi|| \leq ||\chi_\phi||$, since the other direction has already been taken care of.

Suppose $\chi_\phi$ extends to $C^*(S)$. From Proposition~\ref{prop:q-HSS28} we know that $h = \phi - \phi\circ\sigma$ is of trace class, and the function $\phi_0:\Z\to\C$ given by $\phi_0(m-n) = \lim_k \phi(m+k,n+k)$ is well-defined and lies in $B(\Z)$. Also
$$
||\chi_\phi|| = ||Fh||_1 + ||\phi_0||_{B(\Z)}.
$$
Let $ \psi(m,n) = \phi(m,n) - \phi_0(m-n)$, and notice that
\begin{align}\label{eq:psi1}
\psi(m,n) = \sum_{k=0}^\infty h(m+k,n+k) = \Tr(S_{m,n} (Fh) )
\end{align}
by \cite[Lemma 2.2]{HSS}. In the proof of \cite[Theorem 2.3]{HSS} it is shown that there are maps $P_k,Q_k:X\to\ell^2(X)$ such that if $m,n$ are chosen as in Definition~\ref{defi:mn}, then
\begin{align}\label{eq:psi2}
\sum_{k=0}^\infty \la P_k(x),Q_k(y)\ra = \Tr(S_{n,m} (Fh)) \quad (x,y\in X),
\end{align}
and
\begin{align}\label{eq:psi3}
\sum_{k=0}^\infty ||P_k||_\infty ||Q_k||_\infty = ||Fh||_1.
\end{align}
We set
$$
\tilde\phi(x,y) = \phi(m,n) \quad\text{and}\quad \tilde\psi(x,y) = \psi(m,n)
$$
as in Definition~\ref{defi:mn}. Combining \eqref{eq:psi1}, \eqref{eq:psi2} and \eqref{eq:psi3} we see that the function $(x,y)\mapsto \tilde\psi(y,x) = \psi(n,m)$ is a Schur multiplier on $X$ with norm at most $||Fh||_1$. Hence $\tilde\psi$ is also a Schur multiplier with norm at most $||Fh||_1$.

We have
$$
\tilde\phi(x,y) = \tilde\psi(x,y) + \phi_0(m-n),
$$
so by using Lemma~\ref{lem:phi0} we see that $\tilde\phi$ is a Schur multiplier with
$$
||\tilde\phi||_S \leq ||Fh||_1 + ||\phi_0||_{B(\Z)}.
$$
By definition, $||m_{\tilde\phi}|| = ||\tilde\phi||_S$, and since $N_\phi = \beta^{-1} \circ m_{\tilde\phi} \circ\beta$, we conclude that $N_\phi$ is completely bounded with
$$
\| N_\phi \| \leq ||m_{\tilde\phi}|| \leq ||Fh||_1 + ||\phi_0||_{B(\Z)} = \|\chi_\phi\|.
$$

\end{proof}

\subsection{Positive and conditionally negative functions}
Our next goal is to prove analogues of Propositions~\ref{prop:thm1} and~\ref{prop:thm2}. The following proposition characterizes the functions $\phi$ that induce states on the Toeplitz algebra.

\begin{prop}\label{prop:q-thm1}
Let $\phi:\N_0\times\N_0\to\C$ be a function. Then the following are equivalent.
\begin{enumerate}
	\item The functional $\chi_\phi$ extends to a state on the Toeplitz algebra $C^*(S)$.
	\item The multiplier $N_\phi$ extends to a u.c.p. map on the Toeplitz algebra $C^*(S)$.
	\item The functions $F(\phi - \phi\circ\sigma)$ and $G\phi$ are positive definite, and $\phi(0,0) = 1$.
\end{enumerate}
\end{prop}
\begin{proof}
The equivalence (1)$\iff$(3) follows from Proposition~\ref{prop:thm1} and Proposition~\ref{prop:translate} together with Lemma~\ref{lem:FG}. The implication (2)$\implies$(1) follows from the equality $\chi_\phi = \ev_1\circ \pi\circ N_\phi$. So we will only be concerned with (1)$\implies$(2).

Assume $\chi_\phi$ extends to a state on $C^*(S)$. Following the proof of Proposition~\ref{prop:eq-norm} we see that $F(\phi - \phi\circ\sigma)$ is the matrix of a trace class operator, and it is also positive definite. Going through the proof of \cite[Theorem 2.3]{HSS} we make the following observation. When $Fh = F(\phi - \phi\circ\sigma)$ is positive definite, we may choose $P_k = Q_k$, and $\tilde\phi$ becomes a positive definite kernel on the tree $X$. Since $\tilde\phi(x,x) = \phi(0) = 1$ for every $x\in X$, we deduce that $m_{\tilde\phi}$ is u.c.p. Finally, $N_\phi = \beta^{-1} \circ m_{\tilde\phi} \circ\beta$ is also u.c.p., where $\beta$ is the $^*$-isomorphism from $C^*(S) \to C^*(U)$ from before.
\end{proof}

\begin{cor}\label{cor:q4}
Let $\theta:\N_0^2\to\C$ be a function. If $G(\theta - \frac12\theta(0,0))$ is positive definite, and $F(\theta - \theta\circ\sigma)$ is positive definite, then
$$
||N_{e^{-t\theta}}|| \leq 1 \quad \text{for every } t>0.
$$
\end{cor}
\begin{proof}
Observe first that for any $\phi:\N_0^2\to\C$ we have $\|N_{e^{-t\phi}}\| \leq e^{t\| N_\phi\|}$.

Let $\phi = \theta - \frac12\theta(0,0)$. From Proposition~\ref{prop:q-thm1} we deduce that $||N_\phi|| = \phi(0,0)$. Since $\theta = \phi + \phi(0,0)$ we find
$$
\| N_{e^{-t\theta}} \| = e^{-t\phi(0,0)} \| N_{e^{-t\phi}} \| \leq e^{-t\phi(0,0)} e^{t\| N_{\phi} \| } = 1.
$$
\end{proof}

\begin{prop}\label{prop:q-thm2}
Let $\psi:\N_0\times\N_0\to\C$ be a function. Then the following are equivalent.
\begin{enumerate}
	\item For all $t>0$ the function $e^{-t\psi}$ satisfies the equivalent conditions in Proposition~\ref{prop:q-thm1}.
	\item $G\psi$ is a conditionally negative definite kernel, $F(\psi\circ\sigma - \psi)$ is a positive definite kernel, and $\psi(0,0) = 0$.
\end{enumerate}
\end{prop}
\begin{proof}
(1)$\implies$(2): Assume that condition (3) of Proposition~\ref{prop:q-thm1} holds for $e^{-t\psi}$ for each $t > 0$. Then $Ge^{-t\psi}$ is positive definite. It follows that $1- Ge^{-t\psi}$ is a conditionally negative definite kernel, and therefore so is the pointwise limit
$$
\lim_{t\to0}\frac{1 - Ge^{-t\psi}}{t} = \lim_{t\to0}G\left(\frac{1 - e^{-t\psi}}{t}\right) = G\psi.
$$
Since $e^{-t\psi(0,0)} = 1$, we get $\psi(0,0) = 0$. Moreover, $F(e^{-t\psi} - e^{-t\psi\circ\sigma})$ is positive definite by assumption. It follows that the pointwise limit
$$
\lim_{t\to0} F\left( \frac{e^{-t\psi} - e^{-t\psi\circ\sigma}}{t} \right) = F(\psi\circ\sigma - \psi)
$$
is positive definite.

(2)$\implies$(1):
First we note that the set of functions $\phi$ satisfying the equivalent conditions of Proposition~\ref{prop:q-thm1} is closed under products and pointwise limits. Stability under products is most easily established using condition (2), while closure under pointwise limits is most easily seen in condition (3).

By assumption $G\psi$ is conditionally negative definite, so the function $e^{-sG\psi}$ is positive definite for each $s> 0$. We let
$$
\rho_s = G^{-1}\left( \frac{e^{-sG\psi}}{s} \right),
$$
so that $G\rho_s$ is positive definite, and
\begin{align}
\frac1s - \rho_s \to G^{-1}G\psi = \psi
\end{align}
pointwise as $s\to 0$. Using Lemma~\ref{lem:FG} we see that $F(\rho_s - \rho_s\circ\sigma) = G\rho_s - (G\rho_s) \circ\sigma$, so
$$
F(\rho_s - \rho_s\circ\sigma) = \frac{ \left( e^{-sG\psi} - e^{-s(G\psi)\circ\sigma} \right) }{s}= \frac{ e^{-s(G\psi)\circ\sigma} \left( e^{s((G\psi)\circ\sigma - G\psi)} - 1 \right) }{s}.
$$
Since $G\psi$ is conditionally negative definite, so is $(G\psi)\circ\sigma$, and hence $e^{-s(G\psi)\circ\sigma}$ is positive definite. Expanding the exponential function gives
$$
e^{s((G\psi)\circ\sigma - G\psi)} - 1 = \sum_{n=1}^\infty \frac{(sF(\psi\circ\sigma - \psi))^n}{n!}.
$$
Since $F(\psi\circ\sigma - \psi)$ is positive definite, so are its powers and hence the sum in the above equation. It follows that $F(\rho_s - \rho_s\circ\sigma)$ is a product of two positive definite functions and hence itself positive definite.

Looking at Proposition~\ref{prop:q-thm1} we see that $N_{\rho_s}$ is completely positive. It follows that $N_{e^{t\rho_s}}$ is completely positive, so $e^{-t (\frac1s - \rho_s)}$ satisfies the conditions of Proposition~\ref{prop:q-thm1}. Finally, since
$$
e^{-t\psi} = \lim_{s\to 0} e^{-t (\frac1s - \rho_s)},
$$
we conclude that $e^{-t\psi}$ satisfies the conditions of Proposition~\ref{prop:q-thm1}.
\end{proof}

\subsection{The linear bound}

As in the case of $\F_\infty$ we prove that if a kernel $\phi$ satisfies $\|\chi_{e^{-t\phi}}\| \leq 1$ for every $t>0$, then it splits in a useful way. We are also able to compare norms of radial Herz-Schur multipliers on $\F_n$ with norms of functionals on the Toeplitz algebra. These are Theorem~\ref{thm:q-split} and Proposition~\ref{prop:q-norms}.

Recall the definition of the set $\mc S$ from Definition~\ref{defi:S}.

\begin{lem}\label{lem:5}
Let $\phi:\N_0^2\to\C$ be a function. If $G\phi\in\mc S$, then
$$
||N_{e^{-t\phi}}|| \leq 1 \quad\text{for every } t>0.
$$
\end{lem}
\begin{proof}
Suppose $G\phi\in\mc S$ and write $G\phi = G\psi + G\theta$, where
\begin{itemize}
\renewcommand{\labelitemi}{$\cdot$}
	\item $G\psi$ is a conditionally negative definite kernel with $G\psi(0,0) = 0$,
	\item $F(\psi\circ\sigma - \psi)$ is a positive definite kernel,
	\item $G(\theta - \frac12\theta(0,0))$ is a positive definite kernel,
	\item $F(\theta - \theta\circ\sigma)$ is a positive definite kernel.
\end{itemize}
Then also $\phi = \psi + \theta$. From Proposition~\ref{prop:q-thm2} we get that $N_{e^{-t\psi}}$ is a u.c.p. map for every $t>0$, and hence $||N_{e^{-t\psi}}|| = 1$. Also, from Corollary~\ref{cor:q4} we get that $||N_{e^{-t\theta}}|| \leq 1$ for every $t>0$. This combines to show
$$
||N_{e^{-t\phi}}|| \leq ||N_{e^{-t\psi}}|| \ ||N_{e^{-t\theta}}|| \leq 1.
$$

\end{proof}

\begin{lem}
Let $\phi:\N_0^2\to\C$ be a hermitian function. If $\| \chi_\phi \| \leq 1$, then $1 - G\phi \in\mc S$.
\end{lem}
\begin{proof}
Use Proposition~\ref{prop:translate} together with Lemma~\ref{lem:2} and Lemma~\ref{lem:sa}.
\end{proof}

\begin{thm}\label{thm:q-split}
Let $\phi:\N_0^2\to\C$ be a hermitian function. Then $G\phi\in\mc S$ if and only if $\| \chi_{e^{-t\phi}} \| \leq 1$ for every $t>0$.
\end{thm}
\begin{proof}
Suppose $\| \chi_{e^{-t\phi}} \| \leq 1$ for every $t>0$. From the previous lemma we see that $G(1 - e^{-t\chi})$ lies in $\mc S$ for every $t > 0$. Hence, so does $G(1 - e^{-t\chi}) / t$ which converges pointwise to $G\chi$ as $t \to 0$. It now follows from Lemma~\ref{lem:Sclosed} that $G\chi\in\mc S$.

The converse direction is Lemma~\ref{lem:5} combined with Proposition~\ref{prop:eq-norm}.
\end{proof}


\begin{lem}\label{lem:bounded}
Let $h\in\mc V$. Then $h$ is bounded if and only if $F(h)$ is bounded.
\end{lem}
\begin{proof}
Let $h\in \mc V$ be bounded. We prove that $F(h)$ and $F^{-1}(h)$ are bounded. This will complete the proof.

Observe that $\tau(h)$ is bounded with the same bound as $h$. Then $(\id - \tau/q)(h)$ is bounded, so
$$
F^{-1}(h) = \left( 1 - \frac1q \right)^{-1} \left( \id - \frac{\tau}{q} \right)(h)
$$
is also bounded.

Suppose $c\geq 0$ is a bound for $h$. Using \eqref{eq:F} we find
$$
|Fh(m,n)| \leq \left( 1 - \frac1q \right) \sum_{k=0}^\infty \frac{|\tau^k(h)(m,n)|}{q^k} \leq \left( 1 - \frac1q \right) \sum_{k=0}^\infty \frac{c}{q^k} =  c.
$$
This proves that $F(h)$ is bounded as well.
\end{proof}

\begin{prop}\label{prop:q-bound}
If $\phi$ is a Hankel function, and $G\phi\in\mc S$, then $\phi$ is linearly bounded.
\end{prop}
\begin{proof}
Suppose $\phi\in\mc V$ is a Hankel function, and $G\phi\in\mc S$. Let $h = \phi\circ\sigma - \phi$. We wish to prove that $h$ is bounded, since this will give the desired bound on $\phi$. Observe that $h$ is also a Hankel function. We write $h(m,n) = \dot h(m+n)$ for some function $\dot h:\N_0\to\C$.

Since $G\phi$ lies in $\mc S$, there is a splitting of the form $G\phi = G\psi + G\theta$, where
\begin{itemize}
\renewcommand{\labelitemi}{$\cdot$}
	\item $G\psi$ is a conditionally negative definite kernel with $G\psi(0,0) = 0$,
	\item $F(\psi\circ\sigma - \psi)$ is a positive definite kernel,
	\item $G(\theta - \frac12\theta(0,0))$ is a positive definite kernel,
	\item $F(\theta - \theta\circ\sigma)$ is a positive definite kernel.
\end{itemize}
Write $h_1 = \psi\circ\sigma - \psi$ and $h_2 = \theta - \theta\circ\sigma$, and observe that $h = h_1 - h_2$. By the above, Corollary~\ref{cor:4} and Proposition~\ref{prop:thm2} we see that there are vectors $\xi_k,\eta_k$ in a Hilbert space such that
$$
Fh_1(m,n) = \la \eta_m,\eta_n\ra,
\quad
Fh_2(m,n) = \la \xi_m,\xi_n\ra,
$$
for all $m,n\in\N_0$ and
$$
\sum_{k=0}^\infty \|\eta_k - \eta_{k+1}\|^2 < \infty,
\quad
\sum_{k=0}^\infty \|\xi_k\|^2 < \infty.
$$
From this we see that $Fh_2$ is the matrix of a positive trace class operator, and from Lemma~\ref{lem:bounded} we see that $h_2$ is a bounded function. Also, there is $c > 0$ such that $\|\eta_k - \eta_{k+1}\| \leq c$ for every $k$ and $\|\eta_0\| \leq c$, and so we get the linear bound $\|\eta_k\| \leq c(k+1)$. We may even choose $c$ such that $|h_2(m,n)|\leq c^2$ for every $m,n\in\N_0$. From the Cauchy-Schwarz inequality we get
$$
|Fh_1(m,n)| \leq c^2(m+1)(n+1).
$$

We remark that $Fh(0,n) = \left( 1- \frac1q\right) h(0,n)$, so the above with $m = 0$ gives us 
$$
|h_1(0,n)| = \left( 1- \frac1q\right)^{-1} |Fh_1(0,n)| \leq 2c^2(n+1).
$$
Putting all this together gives the linear bound
\begin{align}\label{eq:linear-bound}
|\dot h(n)| = |h_1(0,n) - h_2(0,n)| \leq 2c^2(n+1) + c^2 \leq 2c^2(n+2).
\end{align}

Since $Fh_2 \leq \|Fh_2\| I \leq \|Fh_2\|_1 I = \|h_2\|_1 I$ (as positive definite matrices, where $I$ is the identity operator), we deduce that the function
$$
Fh(m,n) + \|h_2\|_1\delta_{mn} = Fh_1(m,n) + (\|h_2\|_1\delta_{mn} - Fh_2(m,n))
$$
is positive definite. By the Cauchy-Schwarz inequality we have
$$
|Fh(0,n)|^2 \leq (Fh(0,0) + \|h_2\|_1)(Fh(n,n) + \|h_2\|_1)
$$
for every $n\geq 1$, and hence
\begin{align}\label{eq:q-cs}
|\dot h(n)|^2 \leq e(Fh(n,n) + \|h_2\|_1),
\end{align}
where we, in order to shorten notation, have put
$$
e = \left(1 - \frac1q\right)^{-2}\left(Fh(0,0) + \|h_2\|_1\right).
$$
If $e$ is zero, then clearly $\dot h(k) = 0$ when $k\geq 1$. Suppose $e > 0$. Then from \eqref{eq:q-cs} we get
\begin{align*}
\frac{|\dot h(n)|^2}{e} - \|h_2\|_1 &\leq  \sum_{k=0}^n \frac{h(n-k,n-k)}{q^k} \\
&= \sum_{k=0}^n \frac{\dot h(2n-2k)}{q^k} \\
&\leq \frac{q}{q-1}\max\{|\dot h(2n - 2k)| \mid 0\leq k\leq n\} \\
&\leq 2\max\{|\dot h(2k)| \mid 0\leq k\leq n\}.
\end{align*}
In particular we have the following useful observation. Let $a = 2e$ and $b = \|h_2\|_1/2$. Then for each $n\in\N$ there is a $k\leq 2n$ such that
$$
|\dot h(k)| \geq \frac{|\dot h(n)|^2}{a} - b.
$$
We will now show that $|\dot h(n)| \leq 2a+b$ for every $n$. Suppose by contradiction that $|\dot h(n_0)| > 2a + b$ for some $n_0$. We claim that this assumption will lead to the following. For each $m\in\N_0$ there is a $k_m \leq 2^{m+1} n_0$ such that
\begin{align}\label{eq:q-exp}
|\dot h(k_m)| \geq 2^{2^m}(2a+b).
\end{align}
We prove this by induction over $m$. From our observation above we get that there is a $k_0 \leq 2n_0$ such that
$$
|\dot h(k_0)| \geq \frac{|\dot h(n_0)|^2}{a} - b \geq 4a + \frac{b^2}{a} + 4b - b \geq 2(2a + b),
$$
and so \eqref{eq:q-exp} holds for $m = 0$.

Assume that we have found $k_0,\ldots,k_m$. Using our observation we may find $k_{m+1} \leq 2 k_m$ such that
\begin{align*}
|\dot h(k_{m+1})| &\geq \frac{|\dot h(k_m)|^2}{a} - b \geq \frac{(2^{2^m} (2a + b))^2}{a} - b \\
&= 2^{2^{m+1}} \frac{4a^2 + b^2 + 4ab}{a} - b \geq 2^{2^{m+1}} (4a + 4b) - b \\
&\geq 2^{2^{m+1}} (2a +b).
\end{align*}
Finally, note that $k_{m+1} \leq 2k_m \leq 2(2^{m+1}n_0) = 2^{m+2}n_0$ as desired. This proves \eqref{eq:q-exp}. But clearly \eqref{eq:q-exp} is in contradiction with \eqref{eq:linear-bound}, and so we conclude that $|\dot h(n)| \leq 2a + b$ for all $n\in\N_0$. This proves that $\dot h$ is bounded, and hence $\phi$ is linearly bounded.

\end{proof}

In \cite{HSS} the following theorem is proved (Theorem 5.2).
\begin{thm}\label{thm:q-HSS}
Let $\F_n$ be the free group on $n$ generators ($2 \leq n < \infty$), let $\phi:\F_n\to \C$ be a radial function, and let $\dot\phi:\N_0\to\C$ be as in Definition~\ref{defi:radial}. Finally, let $h=(h_{ij})_{i,j\in\N_0}$ be the Hankel matrix given by $h_{ij}=\dot\varphi(i+j)-\dot\varphi(i+j+2)$ for $i,j\in\N_0$. Then the following are equivalent:
	\begin{itemize}
		\item [(i)]$\phi$ is a Herz-Schur multiplier on $\F_n$,
		\item [(ii)]$h$ is of trace class.
	\end{itemize}
	If these two equivalent conditions are satisfied, then there exist unique constants $c_\pm\in\C$ and a unique $\dot\psi:\N_0\to\C$ such that
	\begin{align}\label{eq:q-pres}
		\dot\phi(k)=c_++c_-(-1)^k+\dot\psi(k)\qquad(k\in\N_0)
	\end{align}
	and
	\begin{equation*}
		\lim_{k\to\infty}\dot\psi(k)=0.
	\end{equation*}
	Moreover, with $q = 2n - 1$
	\begin{equation*}
		\|\varphi\|_\HS=|c_+|+|c_-|+	\| Fh\|_1,
	\end{equation*}
	where $F$ is the operator defined by~\eqref{eq:F}.
\end{thm}

\begin{prop}\label{prop:q-norms}
If $\phi:\F_n\to\C$ is a radial function, and $\tilde\phi$ is as in Definition~\ref{defi:radial}, then $\|\chi_{\tilde\phi}\| = \|\phi\|_\HS$.
\end{prop}
\begin{proof}
Let $h_{ij} = \tilde\phi(i,j) - \tilde\phi(i+1,j+1)$. From Theorem~\ref{thm:q-HSS} and Proposition~\ref{prop:q-HSS28} we see that it suffices to consider the case where $h$ is the matrix of a trace class operator, since otherwise $\|\chi_{\tilde\phi}\| = \|\phi\|_\HS = \infty $. If $h$ is of trace class, then we let $\tilde\phi_0(n) = \lim_k \tilde\phi(k+n,k)$. From Theorem~\ref{thm:q-HSS} and Proposition~\ref{prop:q-HSS28} it follows that
\begin{align*}
\|\chi_{\tilde\phi}\| &= \| Fh \|_1 + \|\tilde\phi_0\|_{B(\Z)}, \\
\|\phi\|_\HS &= \| Fh \|_1 + |c_+| + |c_-|,
\end{align*}
where $c_\pm$ are the constants obtained in Theorem~\ref{thm:q-HSS}. It follows from \eqref{eq:q-pres} that
$$
\tilde\phi_0(n) = c_+ + (-1)^n c_-.
$$
Now we only need to see why $|c_+| + |c_-| = \|\tilde\phi_0\|_{B(\Z)}$.

Let $\nu\in C(\mb T)^*$ be given by $\nu(f) = c_+f(1) + c_-f(-1)$ for all $f\in C(\mb T)$. Observe that $\nu(z\mapsto z^n) = c_+ + (-1)^nc_-$. Hence $\nu$ corresponds to $\tilde\phi_0$ under the isometric isomorphism $B(\Z) \simeq C(\mb T)^*$. Hence,
$$
\|\tilde\phi_0\|_{B(\Z)} = ||\nu|| = |c_+| + |c_-|.
$$
This completes the proof.
\end{proof}

\begin{thm}
If $\phi:\F_n\to\R$ is a radial function such that $\|e^{-t\phi}\|_\HS \leq 1$ for each $t> 0$, then there are constants $a,b\geq0$ such that $\phi(x) \leq b + a|x|$ for all $x\in\F_n$. Here $|x|$ denotes the word length function on $\F_n$.
\end{thm}
\begin{proof}
Suppose $\phi:\F_n\to\R$ is a radial function such that $\|e^{-t\phi}\|_\HS \leq 1$ for each $t> 0$, and let $\tilde\phi$ be as in Definition~\ref{defi:radial}. First observe that $\tilde\phi$ is real and symmetric, and hence hermitian. From Proposition~\ref{prop:q-norms} we get that $\| \chi_{e^{-t\tilde\phi}} \| \leq 1$ for every $t > 0$, so Theorem~\ref{thm:q-split} implies that $G(\tilde\phi)\in \mc S$. Since $\tilde\phi$ is a Hankel function, Proposition~\ref{prop:q-bound} ensures that
$$
|\tilde\phi(m,n)| \leq b + a(m+n) \quad\text{for all } m,n\in\N_0
$$
for some constants $a$ and $b$. This implies that
$$
\phi(x) \leq b + a|x| \quad\text{for all } x\in\F_n.
$$
\end{proof}

\section*{Acknowledgements}
The author wishes to thank Uffe Haagerup for suggesting the problems and for valuable discussions on the subject. The author would also like to thank Tim de Laat for many useful comments.


\end{document}